\documentclass[10pt,oneside,a4paper]{scrartcl}

\usepackage[left=3cm,right=4cm,top=3cm,bottom=6cm,includeheadfoot]{geometry}
\usepackage{german, amsmath, amssymb, amsthm, latexsym, fancyhdr, enumerate, txfonts, expdlist, hyperref, mdwlist, mathrsfs, tikz, todonotes,pifont}
\usepackage[active]{srcltx}
\usepackage[T1]{fontenc}
\usepackage[latin1]{inputenc}
\usepackage{subfig}
\newcommand{\w}{\omega}

\newcommand{\X}{\ensuremath{X}}

\renewcommand{\d}{\ensuremath{\partial}}

\newcommand{\mc}{\mathcal}

\newcommand{\alphlist}{\begin{list}{(\alph{enumi})}{\usecounter{enumi}}}
\newcommand{\romanlist}{\begin{list}{(\roman{enumi})}{\usecounter{enumi}}}
\newcommand{\listend}{\end{list}}

\renewcommand{\:}{\colon}
\renewcommand{\=}{\coloneqq}

\newcommand{\ssq}{\ensuremath{\subseteq}}

\newcommand{\eps}{\ensuremath{\varepsilon}}

\newcommand{\lam}{\ensuremath{\lambda}}

\newcommand{\T}{\ensuremath{\mathbb{T}}}

\newcommand{\N}{\ensuremath{\mathbb{N}}} 
\newcommand{\R}{\ensuremath{\mathbb{R}}}
\newcommand{\Z}{\ensuremath{\mathbb{Z}}}

\newcommand{\kreis}{\ensuremath{\mathbb{T}^{1}}}

\newcommand{\I}{\ensuremath{\mathcal I}}
\newcommand{\A}{\ensuremath{\mathcal A}}
\newcommand{\B}{\ensuremath{\mathcal B}}

\newcommand{\nLim}{\ensuremath{\lim_{n\rightarrow\infty}}}

\newcommand{\Leb}{\ensuremath{\operatorname{Leb}}}

\theoremstyle{plain} 
\newtheorem{thm}{Theorem}[section]
\newtheorem{lem}[thm]{Lemma}
\newtheorem{prop}[thm]{Proposition}
\newtheorem{cor}[thm]{Corollary}

\theoremstyle{definition}
\newtheorem{defn}[thm]{Definition}

\theoremstyle{remark}
\newtheorem*{rem}{Remark}

\numberwithin{equation}{section}
\numberwithin{thm}{section}

\makeatletter
\newcommand*{\rom}[1]{\expandafter\@slowromancap\romannumeral #1@}
\makeatother

\title{\Large\textsc{Non-smooth saddle-node bifurcations \rom{2}:
    dimensions of strange attractors}}  \author{ G. Fuhrmann
		\thanks{Department of Mathematics, TU Dresden, Germany. Email: 
		{\tt Gabriel.Fuhrmann@mailbox.tu-dresden.de}}
	\and M. Gr\"oger
		\thanks{Department of Mathematics, Universit\"at Bremen,
		Germany. Email: {\tt groeger@math.uni-bremen.de}} 
	\and T. J\"ager
		\thanks{Department of Mathematics, TU Dresden, Germany. Email: 
		{\tt Tobias.Oertel-Jaeger@tu-dresden.de}}
}
\date{\today}

\begin{document}
\maketitle
\begin{abstract}
  We study the geometric and topological properties of strange
  non-chaotic attractors created in non-smooth saddle-node bifurcations
  of quasiperiodically forced interval maps. By interpreting the
  attractors as limit objects of the iterates of a continuous curve
  and controlling the geometry of the latter, we determine their
  Hausdorff and box-counting dimension and show that these take distinct
  values. Moreover, the same approach allows to describe the
  topological structure of the attractors and to prove their
  minimality.
\end{abstract}

\section{Introduction}

One of the most intriguing phenomena in dynamical systems is the
existence of strange attractors and the fact that these intricate
structures already occur for relatively simple deterministic systems
given by low-dimensional maps and flows. The discovery of paradigm
examples like the H\'enon or the Lorenz attractor has given great
impetus to the field. Usually, strange attractors are associated with
chaotic dynamics. However, this is not always the case, and in a
seminal paper \cite{grebogi/ott/pelikan/yorke:1984} Grebogi, Ott, Pelikan and Yorke demonstrated
that such objects may also occur in systems which do not allow for
chaotic motion -- in the sense of positive topological entropy -- for
structural reasons. Their heuristic and numerical arguments were later
confirmed in a rigorous analysis by Keller \cite{keller:1996}.  The
class of systems considered in \cite{grebogi/ott/pelikan/yorke:1984,keller:1996} were
\emph{quasiperiodically forced (qpf) monotone interval maps}. These are skew
product transformations of the form
\begin{equation}\label{eq: skew-product introduction}
  f\:  \T^d\times X  \to \T^d\times X, \quad 
  (\theta,x)\mapsto (\theta+\w,f_\theta(x)),
\end{equation}
where $\T^d=\R^d/\Z^d$, $X\ssq \R$ is an interval (possibly
non-compact), the rotation vector $\omega\in\T^d$ on the base is
totally irrational and for each $\theta\in\T^d$ the {\em fibre map}
$f_\theta: X\to X$ is a monotone interval map.\footnote{The fact
  that skew product systems of this type do not allow for positive
  entropy follows from an old result of Bowen \cite{Bowen1971}, see
  also \cite{glendinning/jaeger/stark:2009}.}

The specific examples in \cite{grebogi/ott/pelikan/yorke:1984,keller:1996} belong to the class
of so-called {\em pinched skew products}, which are characterised by
the fact that for some $\theta\in \T^d$ the fibre map $f_\theta$ is
constant and consequently the whole fibre $\{\theta\}\times X$ is
mapped to a single point \cite{glendinning:2002}. This greatly
simplifies their analysis, but at the same time it gives them a
certain toy model character. In particular, pinched skew products are
not invertible and can therefore not be the time-one maps of flows,
which are of main interest from the applied point of view.
Notwithstanding, it was later confirmed both numerically
(e.g.~\cite{romeiras/etal:1987,feudel/kurths/pikovsky:1995}) and even
experimentally \cite{Ditto/etal:1990,Ditto/etal:1989} that the
occurrence of strange non-chaotic attractors (SNA's) in systems with quasiperiodic forcing is a widespread
and robust phenomenon, and general methods to rigorously prove their
existence have been established in different settings
\cite{young:1997,bjerkloev:2005a,bjerkloev2009sna,jaeger:2009a,fuhrmann2013NonsmoothSaddleNodesI}.
Thereby, it has turned out that SNA's often play a crucial role in the
bifurcations of invariant curves and often originate from the
collision of these. This pattern for the creation of SNA's has been
named {\em torus collision} or, more specifically, {\em non-smooth
  saddle-node bifurcation}
\cite{heagy/hammel:1994,prasad/negi/ramaswamy:2001,AnagnostopoulouJaeger2012SaddleNodes}.

In contrast to conditions for the existence of SNA's, the structural
properties of these objects are far less understood. From the
mathematical viewpoint, much of the relevant information about the
geometric and dynamical features of an attractor is encoded in
different notions of dimension. Accordingly, the question of computing
dimensions of SNA's has been raised already at an early stage. Based on
numerical evidence, it has been conjectured in
\cite{DingGrebogiOtt1989Dimensions} that the box (or capacity)
dimension of SNA's appearing in different types of qpf systems with one-dimensional base $\kreis$ and one-dimensional
fibres equals two, whereas the information dimension equals one. For
the simple pinched skew products introduced in
\cite{grebogi/ott/pelikan/yorke:1984}, these findings were confirmed
analytically in \cite{jaeger:2007,GroegerJaeger2012SNADimensions}.

The aim here is to perform a similar analysis for SNA's appearing in a
more realistic setting. We concentrate on invertible qpf interval maps
and focus on such SNA's which are created in non-smooth saddle-node
bifurcations. Apart from the dimensions, we obtain the minimality of
the dynamics on the attractors and information about their topological
structure. On an heuristic level, some inspiration is drawn from the
previous work in \cite{jaeger:2007,GroegerJaeger2012SNADimensions}.
Technically, however, the task is considerably more demanding and our
approach builds on a detailed multiscale analysis established in the
first author's article \cite{fuhrmann2013NonsmoothSaddleNodesI}, whose continuation this work presents.
Before stating precise results, we need to introduce some general
notions and a framework for non-smooth saddle-node bifurcations in qpf
interval maps. The latter results from a discrete-time analogue to
work of N\'u\~nez and Obaya on almost periodically forced scalar
differential equations \cite{nunez}, which is provided in
\cite{AnagnostopoulouJaeger2012SaddleNodes}. \medskip

Given $f$ as in (\ref{eq: skew-product introduction}), an {\em
  $f$-invariant graph} is a measurable function $\phi:\T^d\to X$
that satisfies
\begin{equation*}
  f_\theta(\phi(\theta)) \ = \ \phi(\theta+\omega) 
\end{equation*}
for all $\theta\in\T^d$. The associated point set
$\Phi=\{(\theta,\phi(\theta))\mid\theta\in\T^d\}$ is invariant in
this case, and slightly abusing terminology we will refer to both the
function $\phi$ and the set $\Phi$ as an invariant graph. As far as
functions are concerned, we will not distinguish between invariant
graphs that coincide Lebesgue-almost everywhere, and thus implicitly
speak of equivalence classes. By saying an invariant graph has a
certain property, like continuity or semi-continuity, we mean that
there exists a representative in the respective equivalence class
which has this property. The stability of an invariant graph is
determined by its Lyapunov exponent
\begin{equation*}
  \lambda(\phi) \ = \ \int_{\T^d} \log f'_\theta(\phi(\theta)) \ d\theta .
\end{equation*}
If $\lambda(\phi)<0$, then $\phi$ is attracting, in the sense that for
almost every $\theta\in \T^d$ there is $\eps=\eps(\theta)>0$ such that
\begin{equation*}
\left|f^n(\theta,x)-(\theta+n\w,\phi(\theta+n\w))\right|\ \to \ 0
\end{equation*}
for $n\to \infty$ and $x\in B_{\eps}(\phi(\theta))$
\cite{jaeger:2003}. If $\phi$ is continuous, then $\eps$ can be
chosen independent of $\theta\in\T^d$ \cite{sturman/stark:2000}. An
SNA, in this setting, is a non-continuous invariant graph with a
negative Lyapunov exponent. `Strange' here simply refers to the lack
of continuity. We refer to Milnor \cite{milnor:1985} for a broader
discussion of the notion of `strange attractors'.

In the context of forced systems, the significance of invariant graphs
stems from the fact that they are a natural analogue to fixed points of unperturbed
maps, and just like the latter they may bifurcate. As mentioned above,
we will concentrate on saddle-node bifurcations.
In order to keep notation as simple as possible, we may assume without loss of generality
that $[0,1]\ssq \X$ from now on.
We denote by $\mathcal{F}_\omega$ the class of $\mathcal{C}^2$-maps of
the form (\ref{eq: skew-product introduction}) (with fixed rotation
vector $\omega\in \T^d$ in the base). Further, by $\mathcal{P}_\omega$
we denote $\mathcal{C}^2$ one-parameter families in
$\mathcal{F}_\omega$, that is,
\begin{equation*}
    \mathcal{P}_\omega \ = \ \left\{\left. \left(f_\beta\right)_{\beta\in[0,1]}
   \ \right| \ f_\beta \in \mathcal{F}_\omega \textrm{ for all } \beta\in[0,1]
    \textrm{ and } (\beta,\theta,x)\mapsto f_{\beta,\theta}(x)
 \textrm{ is } \mathcal{C}^2 \ \right\} .
\end{equation*}
Elements of $\mathcal{P}_\omega$ will also be denoted by $\hat
f=\left(f_\beta\right)_{\beta\in[0,1]}$. We equip $\mathcal{P}_\omega$
with the $\mathcal{C}^2$-metric and simply refer to the induced
topology as $\mathcal{C}^2$-topology in all of the following. In order
to ensure the occurrence of a saddle-node bifurcation in a prescribed
region $\Gamma=\T^d\times[0,1]$ of the phase space, we need to impose
a number of further conditions. The following assumptions are supposed
to hold for all $\beta\in[0,1]$ and all $\theta\in\T^d$ (if
applicable).
\begin{eqnarray}
  \label{eq:4}
  f_{\beta,\theta}(0) & \leq &  0 \quad \textrm{ and } \quad   f_{\beta,\theta}(1)\ \ \leq \ \ 1 ;\\
  f'_{\beta,\theta}(x) & > & 0 \quad \textrm{ for all } x\in[0,1]  ;\label{e.monotonicity}\\
  f''_{\beta,\theta}(x) & < & 0 \quad \textrm{ for all } x\in[0,1] ; \label{e.concavity}\\
  \textstyle \frac{\partial}{\partial\beta} f'_{\beta,\theta}(x) & < & 0 \quad \textrm{ for all } x\in[0,1] ;\\
  f_0 \textrm{ has } & & \hspace{-2.6eM} \textrm{two continuous invariant graphs in } \Gamma
   \textrm{ and } f_1 \textrm{ has no invariant graph in } \Gamma . \label{eq:5}
\end{eqnarray}
Here, we say $f$ has an invariant graph $\phi$ in $\T^d\times A$ if
$\phi(\theta)\in A$ for all $\theta\in\T^d$. We let
\begin{equation*}
  \mathcal{S}_\omega \ = \ \left\{ \hat f\in\mathcal{P}_\omega \left|\  \hat f 
  \textrm{ satisfies (\ref{eq:4})--(\ref{eq:5})} \ \right.\right\}  . 
\end{equation*}
\begin{thm}[{\cite[Theorem 6.1]{AnagnostopoulouJaeger2012SaddleNodes}}]\label{t.saddle-node}
  Let $\hat f=\left(f_\beta\right)_{\beta\in[0,1]}\in
  \mathcal{S}_\omega$. Then there exists a unique critical parameter
  $\beta_c\in(0,1)$ such that the following holds.  \romanlist
\item If $\beta<\beta_c$, then $f_\beta$ has two invariant graphs
  $\phi^-_\beta<\phi^+_\beta$ in $\Gamma$, both of which are
  continuous. We have $\lambda(\phi^-_\beta)>0$ and $\lambda(\phi^+_\beta)<0$. 
\item If $\beta>\beta_c$, then $f_\beta$ has no invariant graphs in
  $\Gamma$.
\item If $\beta=\beta_c$, then one of the following two possibilities
  hold.
\alphlist
\item[($\mathcal{S}$)] \underline{\em Smooth bifurcation:} $f_{\beta_c}$ has a
  unique invariant graph $\phi_{\beta_c}$ in $\Gamma$, which satisfies
  $\lambda(\phi_{\beta_c})=0$. Either $\phi$ is continuous, or it contains both
  an upper and lower semi-continuous representative in its equivalence
  class.
\item[($\mathcal{N}$)] \underline{\em Non-smooth bifurcation:} $f_{\beta_c}$ has
  exactly two invariant graphs $\phi^-_{\beta_c}<\phi^+_{\beta_c}$ a.e. in
  $\Gamma$. The graph $\phi^-_{\beta_c}$ is lower semi-continuous,
  whereas $\phi^+_{\beta_c}$ is upper semi-continuous, but none of the
  graphs is continuous and there exists a residual set
  $\Omega\ssq\T^d$ such that
  $\phi^-_{\beta_c}(\theta)=\phi^+_{\beta_c}(\theta)$ for all
  $\theta\in\Omega$. \listend \listend
\end{thm}
\begin{rem}
The points in the above set $\Omega$ are called \emph{pinched} points. Due to the semi-continuity, it turns out that $\phi_{\beta_c}^+$ and $\phi_{\beta_c}^-$ are actually continuous in the pinched points (cf. \cite[Lemma~5]{Stark}).
\end{rem}

As said before, the invariant graphs appearing in this statement
have to be understood in the sense of equivalence classes. There is,
however, an intimate relation to the maximal invariant subset of
$\Gamma$, given by 
\begin{equation*}
  \Lambda_\beta \ = \ \bigcap_{n\in\Z} f^n_{\beta}(\Gamma) ,
\end{equation*}
that can be used to obtain well-defined canonical representatives.
This will be important in the statement of our main result.  We write
\begin{equation*}
\Lambda_{\beta,\theta} \  = \  \{x\in[0,1]\mid
(\theta,x)\in\Lambda_\beta\}  .
\end{equation*}
Due to the invariance of $\Lambda_\beta$ and the monotonicity of the
fibre map (\ref{e.monotonicity}), the graphs
\begin{equation}
  \label{eq:8}
  \hat\phi_\beta^-(\theta) \ = \inf \Lambda_{\beta,\theta} \quad\textrm{ and } 
  \quad \hat\phi_\beta^+(\theta) \ = \ \sup\Lambda_{\beta,\theta} 
\end{equation}
are both invariant and thus have to be representatives of
the invariant graphs in part (i) and (iii) of
Theorem~\ref{t.saddle-node}. Moreover, if we write
$\left[\hat\phi^-_\beta,\hat\phi^+_\beta\right]=\left\{(\theta,x)\in\Gamma\mid
\hat\phi^-_\beta(\theta)\leq x \leq
\hat\phi^+_\beta(\theta)\right\}$, then
$\Lambda_\beta=\left[\hat\phi^-_\beta,\hat\phi^+_\beta\right]$.

Theorem~\ref{t.saddle-node} gives a precise meaning to the notion of a
saddle-node bifurcation for a family in $\mathcal{S}_\omega$.
Moreover, it shows that there are two qualitatively different patterns
for such a transition, namely the smooth and the non-smooth case.
While smooth bifurcations can be realised easily by considering direct
products of irrational rotations and suitable interval maps, the
existence of non-smooth bifurcations is much more difficult to establish.
However, as the following result shows, they are nevertheless a
generic case. Recall that $\omega\in\T^d$ is {\em Diophantine} if
there exist $\mathscr C>0$ and $\eta>1$ such that $d(k\w,0)\geq
\mathscr C |k|^{-\eta}$ for all $k\in \Z\setminus\{0\}$.
\begin{thm}[\cite{fuhrmann2013NonsmoothSaddleNodesI}] Let 
  \begin{equation*}
    \mathcal{N}_\omega \ = \ \left\{\hat f\in\mathcal{S}_\omega\mid f_{\beta_c} 
   \textrm{ satisfies } (\mathcal{N}) \right\} \ 
  \end{equation*}
  and suppose $\omega\in\T^d$ is Diophantine. Then
  $\mathcal{N}_\omega$ has non-empty interior in the
  $\mathcal{C}^2$-topology on $\mathcal{P}_\omega$.
\end{thm}
While this statement may seem rather abstract in the above form, it is
important to note that a much more detailed version is given in
\cite{fuhrmann2013NonsmoothSaddleNodesI}. It states that
$\mathcal{N}_\omega$ contains a $\mathcal{C}^2$-open subset
$\mathcal{U}_\omega$ which is completely characterised by a
list of $\mathcal{C}^2$-estimates on the respective parameter
families. However, since this list consists of 16 different and
sometimes rather technical conditions, we refrain from reproducing it
here. A partially intrinsic characterisation that contains all the
information required for our purposes is given in Section~\ref{subsec:
  basic setting and notation}.  In order to fix ideas, readers may
restrict their attention to the following explicit example which
satisfies all the assumptions of our main result below.
\begin{prop}[\cite{fuhrmann2013NonsmoothSaddleNodesI}]
  Let $\omega\in\T^d$ be Diophantine. Then there exists $a_0>0$ such
  that for all $a>a_0$ the parameter family $\hat
  f\in\mathcal{S}_\omega$ given by 
  \begin{equation}
    \label{e.atan-example}
    f_\beta(\theta,x) \ = \ \left(\theta+\omega,\frac{2}{\pi}\arctan(a x) - 
  \beta (1+\cos(2\pi\theta))    \right)
  \end{equation}
  undergoes a non-smooth saddle-node bifurcation, that is, $\hat
  f\in\mathcal{N}_\omega$.
\end{prop}

Our main result now provides information on the geometric and
topological structure of the SNA and the associated ergodic measure
occurring in such non-smooth saddle-node bifurcations. Note that to
each invariant graph $\phi$ an invariant ergodic measure $\mu_\phi$
can be associated by defining
\begin{align*}
\mu_\phi(A)=\Leb_{\T^d}\left(\pi_{\T^d}\left(\Phi\cap A\right)\right),
\end{align*}
where $A\subseteq\T^d\times\X$ is Borel measurable and $\pi_{\T^d}$ is
the canonical projection onto $\T^d$. We denote the box-counting dimension of a
set $A\ssq\T^d\times X$ by $\textrm{D}_B(A)$ and its Hausdorff
dimension by $\textrm{D}_H(A)$. For the explanation of further
dimension-theoretical notions, see Sections~\ref{sec: hausdorff and
  box-counting Dimension} and~\ref{sec: exact dimensional and
  rectifiable measures}.
\begin{thm}\label{t.main}
  Let $\omega\in\T^d$ be Diophantine. Then there exists a set
  $\widehat{\mathcal{U}}_\omega\ssq\mathcal{N}_\omega$ with non-empty
  $\mathcal{C}^2$-interior such that for all $\hat
  f\in\widehat{\mathcal{U}}_\omega$ the SNA $\hat\phi^+_{\beta_c}$ appearing at
  the critical bifurcation parameter satisfies the following.
\romanlist
\item 
  $\textrm{D}_B\left(\hat\Phi^+_{\beta_c}\right)=d+1$ and
  $\textrm{D}_H\left(\hat\Phi^+_{\beta_c}\right)=d$.
\item The measure $\mu_{\phi^+_{\beta_c}}$ is exact dimensional with
  pointwise dimension and information dimension equal to $d$. 
\item The set
  $\Lambda_{\beta_c}=\left[\hat\phi^-_{\beta_c},\hat\phi^+_{\beta_c}\right]$
  is minimal and we have
  $\Lambda_{\beta_c}=\textrm{cl}\left(\hat\Phi^-_{\beta_c}\right)=
  \textrm{cl}\left(\hat\Phi^+_{\beta_c}\right)$.
\item The graph $\hat \phi^+_{\beta_c}$ is the only semi-continuous
  representative in the equivalence class $\phi^+_{\beta_c}$. \listend
  Analogous results hold for the repeller $\phi^-_{\beta_c}$.
  Moreover, for all sufficiently large $a>0$, the parameter
  family $\hat f$ given by (\ref{e.atan-example}) is contained in
  $\widehat{\mathcal{U}}_\omega$.
\end{thm}

Property (iii) has already been considered by M. Herman \cite{herman:1983}. We want to mention that it has been proved previously by
Bjerkl\"ov  for invariant graphs appearing in
quasiperiodic Schr\"odinger cocycles \cite{bjerkloev:2005}, which can be considered a
special case of our setting. Our proof is inspired by that of
Bjerkl\"ov, but puts a stronger focus on the global approximation of
the SNA by iterates of continuous curves. This allows to avoid some
technical complications. The strategy of our proof is outlined at the
beginning of Section~\ref{sec: hausdorff and pointwise dimension}.

We also note that the result on the box-counting dimension is a direct consequence of (iii).
Since the box-counting dimension is stable under taking closures, we have
$\textrm{D}_B\left(\hat\phi^+_{\beta_c}\right) =
\textrm{D}_B(\Lambda_{\beta_c})$. Since the bounding graphs of
$\Lambda_{\beta_c}$ are distinct, this set has positive
$d+1$-dimensional Lebesgue measure and therefore box-counting dimension $d+1$.

\paragraph{Acknowledgements.}
This work was supported by an Emmy-Noether-Grant of the German Research Council
(DFG grant JA 1721/2-1) and is part of the activities of the Scientific Network ``Skew product dynamics and multifractal analysis'' (DFG grant OE 538/3-1).

\section{Preliminaries}

\subsection{Hausdorff and box-counting dimension}\label{sec: hausdorff and box-counting Dimension}
In the following, we recall the definition of the Hausdorff and
box-counting dimension. Further, we state some well known properties
that will be used later on. Suppose $Y$ is a metric space. We denote
the diameter of a subset $A\subseteq Y$ by $|A|$. For $\varepsilon>0$,
we call a finite or countable collection $\{A_i\}$ of subsets of $Y$
an {\em $\varepsilon$-cover} of $A$ if $|A_i|\leq\varepsilon$ for each
$i$ and $A\subseteq\bigcup_i A_i$.
\begin{defn}
For $A\subseteq Y$, $s\geq 0$ and $\varepsilon>0$, we define
\[
\mathcal H_\varepsilon^s(A)\=\inf\left\{\left.\sum\limits_i
    \left|A_i\right|^s \ \right|\ \{A_i\}\text{ is an
    $\varepsilon$-cover of $A$}\right\}
\]
and call
\[
	\mathcal H^s(A)\=\lim\limits_{\varepsilon\to 0} \mathcal H_\varepsilon^s(A)
\]
the \emph{$s$-dimensional Hausdorff measure} of $A$. The
\emph{Hausdorff dimension} of $A$ is defined by
\[
	D_H(A)\=\sup\{s\geq 0 \mid \mathcal H^s(A)=\infty\}.
\]
\end{defn}
The proof of the next lemma is straightforward (cf.\
\cite{GroegerJaeger2012SNADimensions}, for example).
\begin{lem}\label{lem: Hausdorff dimension limsup set}
	Let $A\subseteq Y$ be a $\limsup$ set, meaning that there exists a sequence
	$(A_i)_{i\in\N}$ of subsets of $Y$ with
	\[
        A=\limsup\limits_{i\to\infty}
        A_i\=\bigcap\limits_{i=1}^\infty\bigcup\limits_{k=i}^\infty
        A_k.
	\]
	If $\sum_{i=1}^\infty\left|A_i\right|^s<\infty$ for some
        $s>0$, then $\mathcal H^s(A)=0$ and $D_H(A)\leq s$.
\end{lem}
\begin{lem}[\cite{Pesin1997}]\label{lemma_Hausdorff_dimension_Lipschitz_image}
  Let $Y$ and $Z$ be two metric spaces and assume that $g:A\subseteq
  Y\to Z$ is a bi-Lipschitz continuous map. Then $D_H(g(A))=D_H(A)$.
\end{lem}
\begin{lem}[\cite{Pesin1997}]\label{lem: Hausdorff dimension countably stable}
  The Hausdorff dimension is countably stable, i.e.,
  $D_H\left(\bigcup_i A_i\right)=\sup_i D_H(A_i)$ for any sequence of
  subsets $(A_i)_{i\in\N}$ with $A_i\subseteq Y$.
\end{lem}
\begin{defn}
  The \emph{lower} and \emph{upper box-counting dimension} of a
  totally bounded subset $A\subseteq Y$ are defined as
\begin{align*}
  \underline D_B(A)\=\liminf\limits_{\varepsilon\to 0}\frac{\log N(A,\varepsilon)}{-\log\varepsilon},\\
  \overline D_B(A)\=\limsup\limits_{\varepsilon\to 0}\frac{\log
    N(A,\varepsilon)}{-\log\varepsilon},
\end{align*}
where $N(A,\varepsilon)$ is the smallest number of sets of diameter at
most $\varepsilon$ needed to cover $A$.  If $\underline
D_B(A)=\overline D_B(A)$, then we call their common value $D_B(A)$ the
\emph{box-counting dimension} (or \emph{capacity}) of $A$.
\end{defn}
\begin{rem}
In contrast to the last lemma, we only have that the upper
box-counting dimension is finitely stable. Further,
$D_B(A)=D_B\left(\overline A\right)$. 
\end{rem}
\begin{thm}[\cite{Howroyd1996}]\label{thm: Hausdorff dimension product sets}
  Suppose $Y$ and $Z$ are two metric spaces and consider the Cartesian
  product space $Y\times Z$ equipped with the maximum metric. Then for
  $A\subseteq Y$ and $B\subseteq Z$ totally bounded, we have
\[
	D_H(A\times B)\leq D_H(A)+\overline D_B(B).
\] 
\end{thm}
\subsection{Exact dimensional and rectifiable measures}\label{sec: exact dimensional and rectifiable measures}
We recall the notions of pointwise and information dimension as well as exact dimensional
measures. Further, we provide the definition and some properties of
rectifiable measures where we mainly follow
\cite{AmbrosioKirchheim2000}.

Again, let $Y$ be a metric space. For $x\in Y$, $\varepsilon>0$ let
$B_\varepsilon(x)$ be the open ball around $x$ with radius
$\varepsilon>0$.
\begin{defn}
  Suppose $\mu$ is a finite Borel measure in $Y$. For each point $x$
  in the support of $\mu$ we define the \emph{lower} and \emph{upper
    pointwise dimension} of $\mu$ at $x$ as
\begin{align*}
  \underline d_\mu(x)\=\liminf\limits_{\varepsilon\to 0}\frac{\log\mu(B_\varepsilon(x))}{\log\varepsilon},\\
  \overline d_\mu(x)\=\limsup\limits_{\varepsilon\to
    0}\frac{\log\mu(B_\varepsilon(x))}{\log\varepsilon}.
\end{align*}
If $\underline d_\mu(x)=\overline d_\mu(x)$, then their common value
$d_\mu(x)$ is called the \emph{pointwise dimension} of $\mu$ at $x$.
The \emph{information dimension} of $\mu$ is defined as
\begin{equation*}
  \lim\limits_{\varepsilon\to 0} \frac{\int \log\mu(B_\eps(x))\ d\mu(x)}{\log \eps} , 
\end{equation*}
provided the limit exists. Otherwise, one again defines upper and
lower information dimension via the limit superior and inferior,
respectively.
\end{defn}
\begin{defn}
  We say that the measure $\mu$ is \emph{exact dimensional} if the
  pointwise dimension exists and is constant almost everywhere, i.e.,
  we have
\[
	\underline d_\mu(x)=\overline d_\mu(x)\eqqcolon d_\mu
\]
 $\mu$-almost everywhere.
\end{defn}
\begin{rem}
Note that if $\mu$ is exact dimensional, then in the setting of
separable metric spaces several other dimensions of $\mu$ coincide
with the pointwise dimension \cite{Zindulka2002}. In
particular, this is true for the information dimension
\cite{Young1982Dimension,Pesin1993}.
\end{rem}
\begin{defn}
  For $d\in\N$, we call a Borel set $A\subseteq Y$ \emph{countably
    $d$-rectifiable} if there exists a sequence of Lipschitz
  continuous functions $(g_i)_{i\in\N}$ with $g_i:A_i\subseteq\R^d\to
  Y$ such that $\mathcal H^d(A\backslash\bigcup_i g_i(A_i))=0$. A
  finite Borel measure $\mu$ is called \emph{$d$-rectifiable} if
  $\mu=\Theta\left.\mathcal H^d\right|_A$ for some countably
  $d$-rectifiable set $A$ and some Borel measurable density
  $\Theta:A\to[0,\infty)$.
\end{defn}
Observe that, by the Radon-Nikodym theorem, $\mu$ is $d$-rectifiable
if and only if $\mu$ is absolutely continuous with respect to
$\left.\mathcal H^d\right|_A$ where $A$ is a countably $d$-rectifiable
set.
\begin{thm}[{\cite[Theorem 5.4]{AmbrosioKirchheim2000}}]
  For a $d$-rectifiable measure $\mu=\Theta\left.\mathcal
    H^d\right|_A$, we have
\[
\Theta(x)=\lim\limits_{\varepsilon\to
  0}\frac{\mu(B_\varepsilon(x))}{V_d\varepsilon^d},
\]
for $\mathcal H^d$-a.e.\ $x\in A$, where $V_d$ is the volume of the
$d$-dimensional unit ball. The right-hand side of this equation is
called the $d$-density of $\mu$.
\end{thm}
From the last theorem, we can deduce that the $d$-density exists and
is positive $\mu$-almost everywhere for a $d$-rectifiable measure
$\mu$. This directly implies the next corollary.
\begin{cor}\label{cor: dimensions rectifiable measure}
  A $d$-rectifiable measure $\mu$ is exact dimensional with $d_\mu=d$.
\end{cor}

\subsection{Definition of the set \texorpdfstring{$\widehat{\mathcal{U}}_\omega$}{\widehat{\mathcal{U}}_\omega}}\label{subsec: basic setting and notation}

The aim of this section is to define the set $\widehat{\mathcal{U}}_\omega$ in
Theorem~\ref{t.main}. In principle, it would be possible to work
directly with the set $\mathcal{U}_\omega$ mentioned after
Theorem~\ref{t.saddle-node}, which can be defined in terms of the
explicit $\mathcal{C}^2$-estimates used in
\cite{fuhrmann2013NonsmoothSaddleNodesI}. However, as mentioned we
want to avoid reproducing the somewhat technical characterisation. At
the same time, we have to state a number of facts concerning the
dynamics of the considered parameter families at the bifurcation, which are derived by
means of the multiscale analysis carried out in
\cite{fuhrmann2013NonsmoothSaddleNodesI}.

Hence, what we actually do is to omit all those estimates from
\cite{fuhrmann2013NonsmoothSaddleNodesI} which are only needed to
prove the desired dynamical properties--namely certain slow
recurrence conditions for certain critical sets defined in the
multiscale analysis. Instead, we define $\widehat{\mathcal{U}}_\omega$ as the set
of parameter families which satisfy those
$\mathcal{C}^2$-estimates that are still needed for our purposes and
at the same time show the required dynamical behaviour. This means
that $\widehat{\mathcal{U}}_\omega$ will be defined in a partially intrinsic and
somewhat abstract way. However, the important fact is that it
has non-empty $\mathcal{C}^2$-interior (see Proposition~\ref{prop: hat U w has non-empty interior}) and contains the example \eqref{e.atan-example} for large $a$.
 
In the following, let $f\in \mc F_\w$ be given.
We assume the existence of both an \emph{interval of contraction}
$C=[c,1]\ssq \X$ and \emph{expansion} $E=[0,e]\ssq \X$ where $0<e<c<1$
(the naming becomes clear below) and a closed convex region $\mc I_{0}
\ssq \T^d$, called the \emph{(first) critical region}, such that
\begin{align}
  f_{\theta}\left(x\right)\in C \text{ for all } x\in [e,1] \text{
    and }\theta \notin \I_{0}.
\label{axiom: 2}
\end{align}
Further, we suppose there are $\alpha>1$, $p\geq\sqrt2$ and $S>0$ such that
for arbitrary $\theta,\theta' \in \T^d$ we have
\begin{align}
  \alpha^{-p}|x-x'|\leq |f_\theta(x)-f_{\theta}(x')|&\leq \alpha^{p}
  |x-x'|\ \text{ for all } x,x'\in \X, \label{eq: lipschitz x}\\
  |f_\theta(x)-f_{\theta'}(x)|&\leq S d(\theta,\theta') 
 \text{ for all } x\in \X, \label{eq: lipschitz theta}\\
  |f_\theta(x)-f_{\theta}(x')|&\leq \alpha^{-2/p}|x-x'|\ 
  \label{eq: lipschitz x in C} \text{ for all } x,x'\in C, \\
  |f_\theta(x)-f_{\theta}(x')|&\geq \alpha^{2/p}|x-x'|\
  \label{eq: lipschitz x in E} \text{ for all } x,x'\in E.
\end{align}
These are the explicit estimates needed to define
$\widehat{\mathcal{U}}_\omega$. In order to state the required
dynamical properties, let $K_n=K_0 \kappa^n$ for some integers $\kappa\geq2, \ K_0\in \N$.
Set
\begin{align*}
 b_0\=1, \qquad b_n\=(1-1/K_{n-1})b_{n-1} \ (n\in \N)
\end{align*}
and $b\=\lim_{n\to \infty} b_n$ and assume $K_0$ and $\kappa$ are big enough to ensure that
$b>\sqrt{(p^2+1)/(p^2+2)}$.
Further, let $\left(M_n\right)_{n\in\N_0}$ be a sequence of
integers that satisfies $M_{n}\in[K_{n-1}M_{n-1},2K_{n-1}M_{n-1}-2]$ for
all $n\in \N$, where $M_0\geq 2$.
\begin{defn}
  For $n\in \N_0$, we recursively define the $n+1$-th \emph{critical
    region} $\mc I_{n+1}$ in the following way:
\begin{itemize}
 \item $\mc A_{n} \=\left (\mc I_{n} - (M_n-1)\w\right)\times C$,
 \item $\mc B_{n} \=\left(\mc I_{n} + (M_n+1)\w\right)\times E$,
 \item $\mc I_{n+1}\= \pi_{\T^d} \left(f^{M_n-1} (\A_{n})\cap f^{-(M_n+1)}(\B_{n})\right)$.
\end{itemize}
Note that we trivially have $\I_{n+1}\ssq\I_n$.
For $n\in \N_0$, set $\mc Z^-_n\= \bigcup_{j=0}^{n}
\bigcup_{l=-(M_j-2)}^{0}\mc I_j+l\w$; $\mc Z^+_n\= \bigcup_{j=0}^{n}
\bigcup_{l=1}^{M_j}\mc I_j+l\w$; $\mc V_n \= \bigcup_{j=0}^{n}
\bigcup_{l=1}^{M_j+1}\mc I_j+l\w$; $\mc W_n \= \bigcup_{j=0}^{n}
\bigcup_{l=-(M_j-1)}^{0}\mc I_j+l\w$. Moreover, set $\mc V_{-1},\mc W_{-1}=\emptyset$.
\end{defn}

\begin{defn}\label{def: recurrence properties of critical intervals}
  Let $n\in \N_0$. For $c_0>0$, set $\eps_n \= c_0 \alpha^{-M_{n-1}\cdot b/(2p)}$, where we put $M_{-1}=0$ for convenience. We say $f$
  verifies $(\mc F1)_n$ and $(\mc F2)_n$, respectively if $\I_j\neq \emptyset$ and
\begin{enumerate}[$(\mc F1)_n$]
\item $d\left(\mc I_{j}, \bigcup_{k=1}^{2K_jM_j} \mc I_{j}
    +k\w\right)>\eps_j$, \label{axiom: diophantine 1}
 \item $\left(\mc I_{j}- (M_j-1)\w \cup \mc I_{j}+ (M_j+1)\w\right) \bigcap
\left ( \mc V_{j-1} \cup \mc W_{j-1}\right)
=\emptyset$ \label{axiom: diophantine 2}
\end{enumerate}
for $j=0,\ldots,n$ and $n\in \N_0$. If $f$ satisfies both
$(\mc F1)_n$ and $(\mc F2)_n$, we say $f$ satisfies
$(\mc F)_n$. Further, we say $f$ satisfies $(\mathcal{E})_n$ if
\begin{enumerate}[$(\mathcal{E})_n$]
\item\quad  $|\I_n|\ < \eps_n$, \label{axiom: I_n}
\end{enumerate}
where $|\I_n|$ denotes the diameter of $\I_n\ssq\T^d$.
\end{defn}

In the following, we say $f$ \emph{satisfies
(\ref{axiom: 2})--(\ref{eq: lipschitz x in E}), $(\mathcal{F})_n$ and $(\mathcal{E})_n$} if it verifies the respective assumptions for some choice of the above constants.
With these notions, we are now in the position to define the set
$\widehat{\mathcal{U}}_\omega$.
\begin{defn}
For $\omega\in\T^d$,
set
  \begin{equation*}
    \widehat{\mathcal{U}}_\omega \ = \ \left\{ \hat f\in\mathcal{S}_\omega\mid 
     f_{\beta_c} \textrm{ satisfies (\ref{axiom: 2})--(\ref{eq: lipschitz x in E}), } 
   (\mathcal{F})_n \textrm{ and } (\mathcal{E})_n \textrm{ for all } n\in\N \right\} .
  \end{equation*}
\end{defn}

The following result is now contained implicitly in
\cite{fuhrmann2013NonsmoothSaddleNodesI}, see \cite[Theorem
4.18]{fuhrmann2013NonsmoothSaddleNodesI} and its proof.

\begin{prop}[\cite{fuhrmann2013NonsmoothSaddleNodesI}]\label{prop: hat U w has non-empty interior}
  For Diophantine $\omega\in\T^d$, the set $\widehat{\mathcal{U}}_\omega$ has non-empty
  $\mathcal{C}^2$-interior and we have
  $\widehat{\mathcal{U}}_\omega\ssq\mathcal{N}_\omega$. Moreover, for all sufficiently large $a>0$, the parameter
  family $\hat f$ given by (\ref{e.atan-example}) is contained in
  $\widehat{\mathcal{U}}_\omega$.
\end{prop}

Thus, in order to prove Theorem~\ref{t.main}, our only task is to show that
the properties of the parameter families in
$\widehat{\mathcal{U}}_\omega$ stated in this section imply the
assertions on the dimensions and the topological structure of
$\phi^+_{\beta_c}$ and $\phi^-_{\beta_c}$.

\section{Hausdorff, pointwise and information dimension}

\label{sec: estimates on the iterated boundary lines}

Our analysis of the structure of the SNA $\hat\Phi^+_{\beta_c}$
appearing in parameter families $\hat f\in\widehat{\mathcal{U}}_\omega$
hinges on the fact that the function $\hat\phi^+_{\beta_c}$ can be
approximated by the images of the curve $\T^d\times\{1\}$ under
successive iterates of the map $f_{\beta_c}$. Since from now on the
critical parameter $\beta_c$ and thus also the map $f_{\beta_c}$ are
fixed, we suppress the parameter from the notation. Hence, from now on
$f$ will always denote a map that belongs to 
\begin{align*}
 \mc V =\left\{f \in
\mathcal{F}_\omega\mid f \text{ satisfies 
(\ref{eq:4})--(\ref{e.concavity}), (\ref{axiom: 2})--(\ref{eq:
  lipschitz x in E})  as well as } (\mathcal{F})_n \text{ and }
(\mathcal{E})_n \text{ for all } n\in\N
\right\}.
\end{align*}
As before, we let 
\begin{equation*}
  \Lambda \ = \ \bigcap_{n\in\Z} f^n(\Gamma)
\end{equation*}
be the maximal $f$-invariant set inside $\Gamma$ and denote by
$\phi^-$ and $\phi^+$ its bounding graphs, that is,
$\phi^-(\theta)=\inf\Lambda_\theta$ and
$\phi^+(\theta)=\sup\Lambda_\theta$ (cf. \eqref{eq:8}). Now given $\theta \in \T^d$, let
\begin{align*}
  \phi_n^+(\theta)\= f_{\theta-n\w}^n(1)=f_{\theta-\w}\circ\ldots\circ f_{\theta-n\w}(1) \quad \text{and}
  \quad \phi_n^-(\theta)\= f_{\theta+n\w}^{-n}(0)=f^{-1}_{\theta+\w}\circ\ldots\circ f_{\theta+n\w}^{-1}(0),
\end{align*}
with $f^n_{\theta}(x)=\pi_x \circ f^n(\theta,x)$ for all integers\footnote{
Note that the invariant graph $\phi^-\geq 0$ cannot be crossed by any orbit. Hence, due to the monotonicity of $f_\theta$ on all $\X$ (for each $\theta$) as well as
\eqref{eq:4} and \eqref{e.monotonicity}, $f^{-n}_\theta(0)$ is indeed well-defined for all $n\in\N$ and arbitrary $\theta\in \T^d$.
} 
$n\in\Z$ where $\pi_x$ is the projection to the second coordinate.
We call $\phi_n^+$ the \emph{$n$-th iterated upper boundary line} and
$\phi_n^-$ the \emph{$n$-th iterated lower boundary line}.  Assumption (\ref{eq:4}) and the monotonicity \eqref{e.monotonicity} yield that $\left(\phi_n^+\right)_{n\in
  \N}$ and $\left(\phi_n^-\right)_{n\in \N}$ are monotonously
decreasing and increasing, respectively.  Moreover, it is easy to see
from \eqref{e.monotonicity} that
$[\phi^-_n,\phi^+_n]=\bigcap_{k=-n}^n f^n(\Gamma)$. As a consequence, it
is immediate that 
\begin{equation*}
  \phi^+(\theta) \ = \ \nLim \phi^+_n(\theta) \quad \textrm{ and } 
  \quad \phi^-(\theta) \ = \nLim \phi^-_n(\theta)  . 
\end{equation*}
Thus, in order to draw conclusions on the structure of the bounding graphs, it is natural to study the iterated boundary lines first.
Figure~\ref{fig: iterated boundary lines} shows the first $6$ iterated boundary lines for the critical parameter in the example family (\ref{e.atan-example}) with $\w$ the golden mean and
parameters $a=40$ and $\beta_c\approx 0.48714$. These pictures reveal a
very characteristic pattern. Let us look carefully at the evolution of $\phi^+_n$.

\begin{figure}
\centering
\begin{tabular}{ccc}
\subfloat{\includegraphics[scale=0.2]{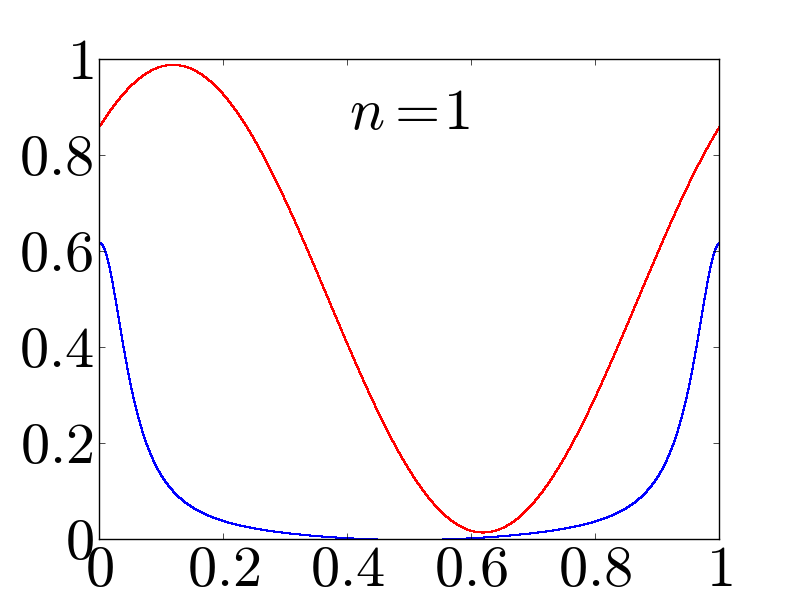}} &
\subfloat{\includegraphics[scale=0.2]{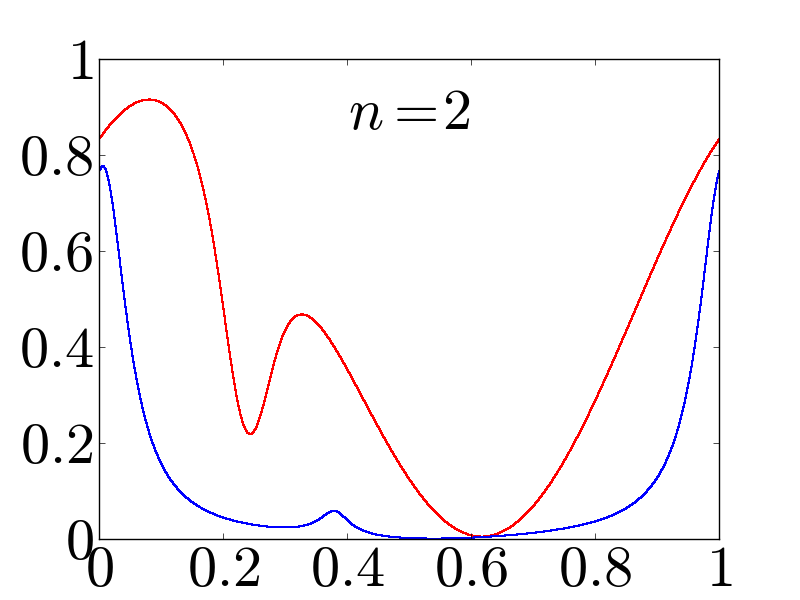}} &
\subfloat{\includegraphics[scale=0.2]{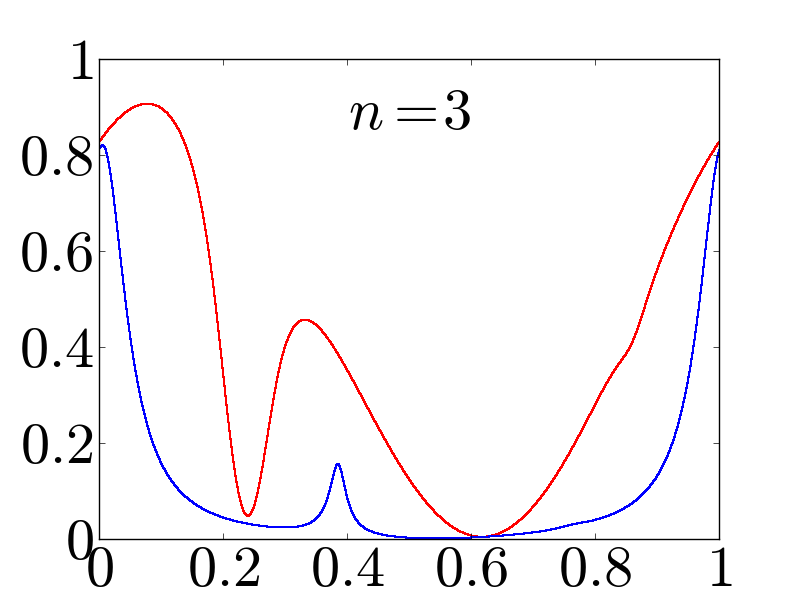}} \\
\subfloat{\includegraphics[scale=0.2]{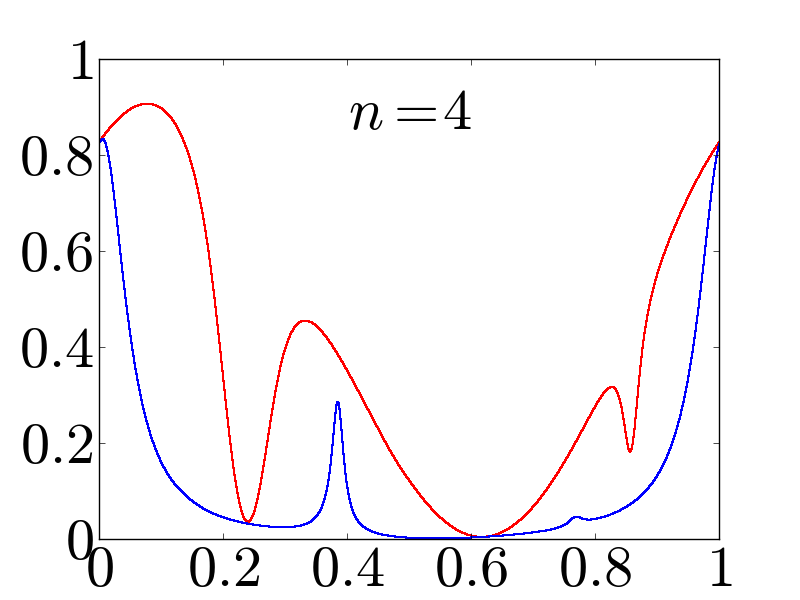}} &
\subfloat{\includegraphics[scale=0.2]{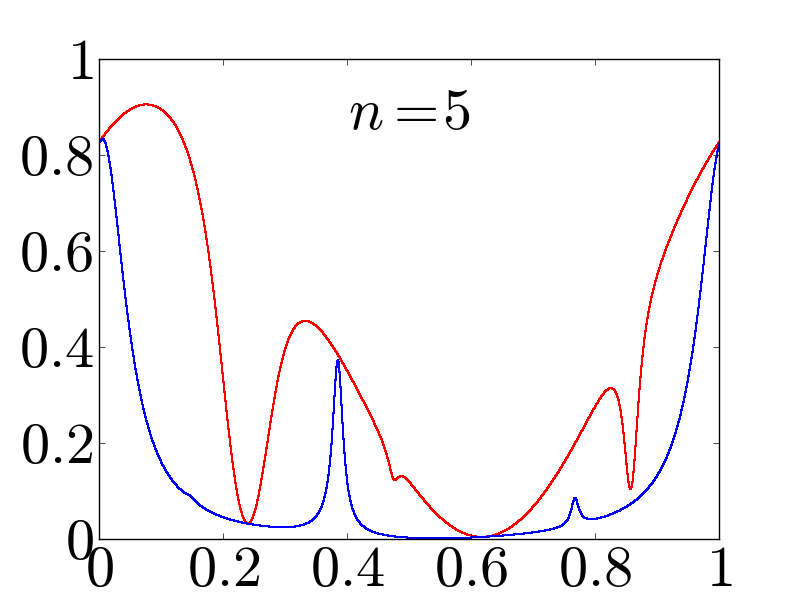}} &
\subfloat{\includegraphics[scale=0.2]{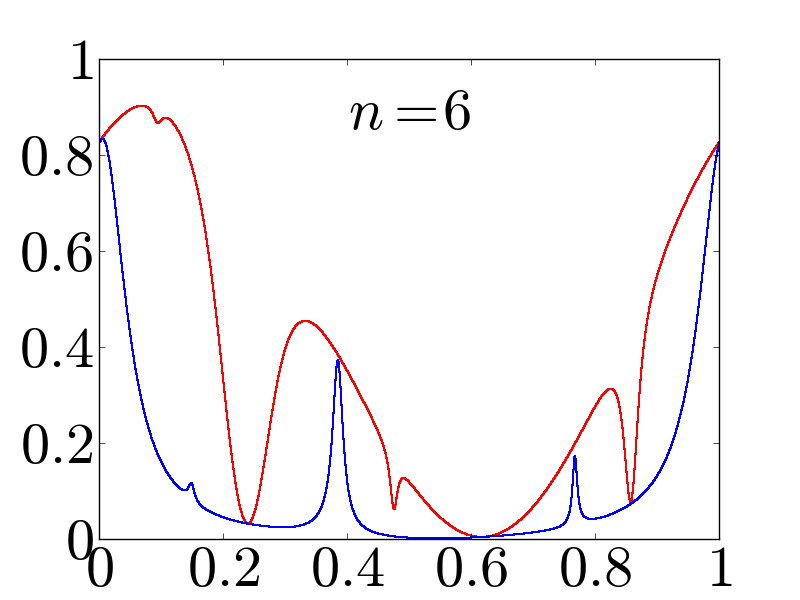}} \\
\end{tabular}
\caption{The first $6$ iterated upper and lower boundary lines $\phi^+_n$ (red) and $\phi^-_n$ (blue), respectively, of the family \eqref{e.atan-example} for $a=40$ at $\beta=0.48714$ with $\w$ the golden mean.}
\label{fig: iterated boundary lines}
\end{figure}

For $n=1$, we see that a first peak exists in the vicinity of $\theta=\w$, that is, above the set $\I_0+\w$ (cf. \eqref{axiom: 2}). After a second iteration, the image of this peak appears as a second peak in the vicinity of $2\w$ while outside this new peak the graph seems--more or less--unchanged. The second peak is not as pronounced as the first peak yet since the strong expansion close to the zero line (due to \eqref{eq: lipschitz x in E}) enlarged the tiny gap between $\phi^+_1(\w)$ and $\phi^-_1(\w)$. However, after one more iteration, the second peak is \emph{stabilised}, that is, its shape is essentially fixed for higher iterations. It is also important to observe
that the graph outside this peak has not changed apart from a small neighbourhood of $3\w$ in the step from $n=2$ to $n=3$. Furthermore, note that the second peak 
is of much smaller size than the first one.

Though the third peak around $3\w$ is already hardly visible at $n=3$, it clearly stabilises until $n=6$ and the graph only changes close to $4\w$ and $5\w$ along this stabilisation. Altogether, this motivates the following qualitative claim.

\begin{center}
\begin{minipage}{0.9\textwidth}
\emph{
$\phi^+_{n+1}$ differs from $\phi^+_{n}$ only in smaller and smaller neighbourhoods of those peaks around $j\w$ (for $j=1,\ldots,n+1$) which are not stabilised yet after $n$ iterations.
}
\end{minipage}
\end{center}

The point is that every peak eventually stabilises in those $\theta$ which are not hit by peaks that appear at higher iterations. Moreover, the measure of the these future peaks tends to zero.
As $\phi_j^+$ is Lipschitz-continuous with a Lipschitz constant $L_j$, 
the claim implies that we get essentially the same Lipschitz constant $L_j$ for $\phi_n^+$ (with arbitrary $n\geq j$) at all those points at which $\phi_j^+$ is stabilised already. 

By this means, we are able to establish a decomposition of $\phi^+$ into Lipschitz graphs whose Hausdorff dimension equals $d$ (see Lemma~\ref{lemma_Hausdorff_dimension_Lipschitz_image}).  By the countable stability of the Hausdorff dimension (see Lemma~\ref{lem: Hausdorff dimension countably stable}), this yields that $D_H(\Phi^+)=d$. 
Part (iii) and (iv) of Theorem~\ref{t.main} are not so easy to illustrate on this qualitative level since we need some understanding of the local densities of those sets which are not hit by future peaks. Still, despite some refinement, the arguments are very much based on the above observations.

To formalise ideas, we introduce
\begin{align*}
\Omega_j^n\= \T^d \setminus \bigcup_{k=j}^{\infty} \bigcup_{l=M_{k-1}}^{\min\{n,2K_kM_k\}} \I_k+l\w, \qquad
\Omega_j\= \bigcap_{n\in\N}\Omega_j^n = \T^d \setminus \bigcup_{k=j}^{\infty}
\bigcup_{l=M_{k-1}}^{2K_kM_k}\I_k+l\w,
\end{align*}
where $j,n \in \N$. 
A way to interpret these definitions in terms of our qualitative discussion is the following: by the recursive definition of $\I_j$ (cf. Section~\ref{subsec: basic setting and notation}),
the size of the $M_{j-1}$-th peak is about $|\I_j|$.
Hence, $\Omega_j$ only contains points which are not hit by any peak that appears
after $M_{j-1}$ iterations. Likewise, $\Omega_j^n$ contains points at which $\phi_n^+$ might stabilise in finite time, but at which new peaks could still appear at future iterations.

Observe that $K_kM_k\leq K_0 \kappa^k \cdot 2 K_{k-1}M_{k-1}\leq \ldots\leq K_0^{k+1} \kappa^{\sum_{l=1}^{k} l} 2^{k} M_0$ while $|\I_k|<\eps_k=c_0\alpha_0^{-M_{k-1}}\leq c_0 \alpha_0^{-K_0^{k-1} \kappa^{\sum_{l=1}^{k-2} l} 2^{k-1} M_0}$. Thus,
we have $2K_{k}M_{k}\eps_{k}^d<\eps_{k}^{d/2}$
for large enough $k$
and hence,
\begin{align}\label{eq: measure of Omegaj}
\operatorname{Leb}_{\T^d}\left(\bigcup_{k=j}^\infty\bigcup_{l=M_{k-1}}^{2K_kM_k}\I_k+l\w\right)<
 \sum_{k=j}^\infty V_{d}2K_kM_k\eps_k^d <\sum_{k=j}^\infty V_{d}\eps_k^{d/2},
\end{align}
for large enough $j$, where $V_{d}$ is the normalising factor of the $d$-dimensional Lebesgue measure.
Thus, $\operatorname{Leb}_{\T^d} (\Omega_j)>0$ for large enough $j\in \N$.

There might still be points which get hit by infinitely many peaks so that no eventual stabilisation occurs. These are collected within
\begin{align*}
 \Omega_\infty \= \T^d\setminus \bigcup_{j\in \N} \Omega_j=\bigcap_{i=1}^{\infty} 
\bigcup_{k=i}^{\infty}
\bigcup_{l=M_{k-1}}^{2K_kM_k}\I_k+l\w.
\end{align*}

In the following, we only consider the upper boundary lines $\phi^+_n$ and the upper bounding graph $\phi^+$. All of the
results and proofs which are only stated in terms of $\phi^+$ and $\phi_n^+$ hold analogously for the lower boundary lines $\phi_n^-$ and the lower bounding graph $\phi^-$
as can be seen by considering $f^{-1}$ instead of $f$.

The next proposition is the basis of our geometrical investigation of $\phi^+$. 
Its proof, which is the technical core of this paper, is given in the last section. However, the statement should seem plausible to the reader in the light of the above discussion.
\begin{prop}\label{prop: properties iterated boundary lines}
Let $f\in \mc V$. There are $\lam>0$ and $C>0$ such that the following is true for sufficiently large $j$.
\begin{enumerate}[(i)]
\item 
Suppose $\theta \in \Omega^n_j$ and $n >2K_{j-1}M_{j-1}-M_{j-1}-1$. Then
$|\phi^+_n(\theta)-\phi^+_{n-1}(\theta)|\leq \alpha^{-\lam (n-1)}$.
\label{prop: properties iterated boundary lines ii} 
\item 
Suppose $\theta,\theta'\in \Omega_j^n$ and $n\in \N$.
Then $\left|\phi^+_{n}(\theta)-\phi^+_{n}(\theta')\right|\leq L_j
  d(\theta,\theta')$ for some $L_j\leq\eps_j^{-CK_{j-1}}$ independent of $n$.
\label{prop: properties iterated boundary lines iii}
\end{enumerate}
\end{prop}

\label{sec: hausdorff and pointwise dimension}
Now, this information on the geometry of the curves $\phi_n^+$ allows
to determine the Hausdorff dimension of $\Phi^+$ rather easily
(cf. \cite{GroegerJaeger2012SNADimensions}).
\begin{thm}\label{prop: dimensions_subgraphs}
Suppose $f\in\mathcal{V}$. Then the following statements hold:
\begin{enumerate}[(i)]
	\item $D_H(\Phi^+)=d$,
	\item $\mu_{\phi^+}$ is $d$-rectifiable and exact dimensional
          with $d_{\mu_{\phi^+}}=d$.
\end{enumerate}
\end{thm}
\begin{proof}
  For each $j\in\N\cup\{\infty\}$ set
  $\psi_j\=\left.\phi^+\right|_{\Omega_j}$.  First, we want to show
  that the graph $\Psi_j=\{(\theta,\psi_j(\theta))\:\theta\in\Omega_j\}$ is the image of a bi-Lipschitz continuous
  function $g_j$ for all $j\in\N$. Define
  $g_j:\Omega_j\to\Omega_j\times\X$ via
  $\theta\mapsto(\theta,\psi_j(\theta))$ for all
  $j\in\N\cup\{\infty\}$. We have that $g_j(\Omega_j)=\Psi_j$ and
  $d_{\T^d\times\X}(g_j(\theta),g_j(\theta'))\geq d(\theta,\theta')$
  for all $\theta,\theta'\in\Omega_j$. 
We may assume without loss of generality that $j$ is large enough\footnote{Observe that for $j\leq J$, we have $\Psi_j\ssq\Psi_{J}$ because $\Omega_j\ssq\Omega_{J}$.} so that
Proposition~\ref{prop: properties
    iterated boundary lines} \eqref{prop: properties iterated boundary lines iii} yields that
  $\left.\phi_n\right|_{\Omega_j}$ is Lipschitz continuous with
  Lipschitz constant $L_j$ independent of
  $n$. Since
  $\psi_j=\lim_{n\to\infty}\left.\phi_n\right|_{\Omega_j}$, we also
  get that $\psi_j$ is Lipschitz continuous with the same constant and
  therefore
\begin{align*}
		d_{\T^d\times\X}(g_j(\theta),g_j(\theta'))\leq\left(1+L_j\right)d(\theta,\theta'),
\end{align*}
for all $\theta,\theta'\in\Omega_j$ and $j\in\N$.  
Hence, $g_j$ is bi-Lipschitz continuous for each $j\in\N$.

(i) We want to make use of the fact that the Hausdorff dimension is
countably stable, see Lemma \ref{lem: Hausdorff dimension countably
  stable}. Because of the bi-Lipschitz continuity, we get that
$D_H(\Psi_j)=D_H(\Omega_j)$. Since $\Leb_{\T^d}(\Omega_j)>0$ for large
enough $j$, this implies $D_H(\Psi_j)=d$.  What is left to show is
that $D_H(\Psi_\infty)\leq d$. Observe that $\Omega_\infty$ is a
$\limsup$ set. With a proper relabelling and doing a similar estimation
as in \eqref{eq: measure of Omegaj}, we can use Lemma \ref{lem:
  Hausdorff dimension limsup set} to conclude that
$D_H(\Omega_\infty)\leq s$ for all $s>0$. Therefore,
$D_H(\Omega_\infty)=0$. Further,
$\Psi_\infty\subset\Omega_\infty\times\X$ and hence
$D_H(\Psi_\infty)\leq D_H(\Omega_\infty)+D_B(\X)=1\leq d$, applying
Theorem \ref{thm: Hausdorff dimension product sets}.

(ii) Note that by definition, $\mu_{\phi^+}$ is absolutely continuous
with respect to $\left.\mathcal H^d\right|_{\Phi^+}$. We have that
$\mu_{\phi^+}(\Psi_\infty)=0$ and therefore $\mu_{\phi^+}$ is also
absolutely continuous with respect to $\left.\mathcal
  H^d\right|_{\Phi^+\backslash\Psi^\infty}$. Since
$\Phi^+\backslash\Psi_\infty=\bigcup_{j\in \N}\Psi_j$ is a
countably $d$-rectifiable set--using the observation from the
beginning of the proof--we get that $\mu_{\phi^+}$ is $d$-rectifiable,
too. Now, by applying Corollary \ref{cor: dimensions rectifiable
  measure}, we obtain that $\mu_{\phi^+}$ is exact dimensional with
pointwise dimension $d_{\mu_{\phi^+}}=d$.
\end{proof}
\begin{rem}
By the remark in Section~\ref{sec: exact dimensional and rectifiable measures}, we immediately get that the information dimension of $\mu_{\phi^+}$ equals $d$.
\end{rem}
\section{Minimality and box-counting dimension}\label{sec:
  filled-in property}
For $n\in \N_{0}$, we denote by $\tilde \I_n$ the $\eps_{n}/2$-neighbourhood of $\I_n$, that is, $\tilde \I_n\=\bigcup_{\theta \in \I_n}B_{\eps_{n}/2}(\theta)$, where
$B_{r}(\theta)$ denotes the open ball of radius $r$ centred at $\theta$. Set
\begin{align*}
 \tilde \Omega_\infty \= \bigcap_{j=1}^\infty \bigcup_{k=j}^\infty \bigcup_{l=M_{k-1}}^{2K_kM_k} \tilde \I_k+l\w.
\end{align*}
\begin{lem}\label{lem: theta NOT in tilde Omega infty}
Suppose $\theta\notin\tilde \Omega_\infty$. Then there exists $j_0\in\N$ such that for all integers $j\geq j_0$ we have
$\theta \in \Omega_j$ and
\begin{align}\label{eq: almost every point in a neighbourhood of Omegaj is in Omegaj}
 \operatorname{Leb}_{\T^d}(B_{\eps_n/2}(\theta)\cap\Omega_j)/\operatorname{Leb}_{\T^d}(B_{\eps_n/2}(\theta))\to 1,
\end{align}
for $n\to \infty$.
\end{lem}
\begin{proof}
By the assumptions, there is $j_0\in\N$ such that
$\theta \notin \bigcup_{k=j_0}^\infty\bigcup_{l=M_{k-1}}^{2K_kM_k} \tilde \I_k+l\w$. Fix an arbitrary $j\geq j_0$ and observe that
\begin{align*}
B_{\eps_n/2}(\theta)\cap \left(\bigcup_{k=j}^n\bigcup_{l=M_{k-1}}^{2K_kM_k}\I_k+l\w\right)
=\emptyset
\end{align*}
for $n\geq j$ by definition of $\tilde \I_k$.
Thus,
\begin{align*}
B_{\eps_n/2}(\theta)\cap\Omega_j=B_{\eps_n/2}(\theta)\setminus\bigcup_{k=j}^\infty\bigcup_{l=M_{k-1}}^{2K_kM_k}\I_k+l\w=B_{\eps_n/2}(\theta)\setminus\bigcup_{k=n+1}^\infty\bigcup_{l=M_{k-1}}^{2K_kM_k}\I_k+l\w.
\end{align*}
Similarly as in \eqref{eq: measure of Omegaj}, we get
$\operatorname{Leb}_{\T^d}\left(\bigcup_{k=n+1}^\infty\bigcup_{l=M_{k-1}}^{2K_kM_k}\I_k+l\w\right)<\sum_{k=n+1}^\infty V_{d}\eps_k^{d/2}$ for large enough $n$, where $V_{d}$ normalises the Lebesgue measure.
\end{proof}

\begin{lem}\label{lem: theta in tilde Omega infty}
 Suppose $\theta \in \tilde \Omega_\infty$. For each $\ell \in \N$, there are arbitrarily large $j$ such that
\begin{align}\label{eq: theta in tilde omega infty}
B_{\eps_{j}/2}(\theta) \ssq \Omega_{j+1}^{2K_{j+\ell}M_{j+\ell}}
\end{align}
and
\begin{align}\label{eq: theta in tilde omega infty then big neighbourhood is in omega j}
\operatorname{Leb}_{\T^d} \left(B_{\eps_{j}/2}(\theta)\right)- \operatorname{Leb}_{\T^d} \left(B_{\eps_{j}/2}(\theta)\cap \Omega_{j+1}\right) <\eps_{j+\ell}.
\end{align}
\end{lem}
\begin{proof}
For $n\in \N$, we define
\begin{align*}
 j_n \= 
\max\left\{p\in \N_{0}\:\exists l \in \left[M_{p-1},\min\left\{n,2K_p M_p\right\}\right] \text{ such that } \theta\in\tilde\I_p+l\w \right\}
\end{align*}
and let $l_n \in\left[M_{j_n-1}, 2K_{j_n}M_{j_n}\right]$ be the corresponding time such that $\theta\in\tilde\I_{j_n}+l_n\w$,
where uniqueness is guaranteed by $(\mc F1)_{j_n}$. Note that $j_n$ and $l_n$ are well-defined for sufficiently large $n$ and $j_n \stackrel{n\to \infty}{\longrightarrow}\infty$ because $\theta\in \tilde\Omega_\infty$.

Further, let $\theta_\ast \in \bigcap_{n=0}^\infty \I_n$.
Note that
$d(\theta_\ast+l\w,\theta)<\frac32 \eps_{j_n}$ for all $l$ for which $\theta\in \tilde \I_{j_n}+l\w$. Now, suppose there is $k\in \N$ such that $\theta\in\tilde\I_{j_n}+l_n\w+k\w$.
Then 
\begin{align*}
d\left(k\w,0\right)= d\left(\theta_\ast+(l_n+k)\w,\theta_\ast+l_n\w\right)\leq
d\left(\theta_\ast+(l_n+k)\w,\theta\right)+ d\left(\theta,\theta_\ast+l_n\w\right)<3 \eps_{j_n}.
\end{align*}
As $\w$ is Diophantine, this means $\mathscr C |k|^{-\eta}<d(k\w,0)<3 \eps_{j_n}$ and hence 
\begin{align}\label{eq: diophantine implies long return times}
 |k|>\tilde c \eps_{j_n}^{-1/\eta},
\end{align}
where $\tilde c>0$. Define
\begin{align*}
 J_n\=\max\left\{N\:2K_NM_N< \tilde c \eps_{j_n}^{-1/\eta}\right\}.
\end{align*}
By \eqref{eq: diophantine implies long return times}, we have
\begin{align*}
 B_{\eps_{j_n}/2}(\theta)\ssq\Omega_{j_n+1}^{2K_{J_n}M_{J_n}}.
\end{align*}
Since $j_n/J_n \stackrel{n\to \infty}{\longrightarrow} 0$,
we have thus shown that for any $\ell\in \N$ there is arbitrarily large $j$ such that $B_{\eps_j/2}(\theta) \ssq \Omega_{j+1}^{2K_{j+\ell}M_{j+\ell}}$.

Given $\ell \in \N$, assume $j$ is such that \eqref{eq: theta in tilde omega infty} holds. Then,
\begin{align*}
B_{\eps_{j}/2}(\theta)\cap\Omega_{j+1}=B_{\eps_{j}/2}(\theta)\setminus\bigcup_{k=j+1}^\infty\bigcup_{l=M_{k-1}}^{2K_kM_k}\I_k+l\w=B_{\eps_{j}/2}(\theta)\setminus\bigcup_{k=j+\ell+1}^\infty\bigcup_{l=M_{k-1}}^{2K_kM_k}\I_k+l\w.
\end{align*}
Finally,
$\operatorname{Leb}_{\T^d}\left(\bigcup_{k=j+\ell+1}^\infty\bigcup_{l=M_{k-1}}^{2K_kM_k}\I_k+l\w\right)<\sum_{k=j+\ell+1}^\infty V_{d}\eps_k^{d/2}< \eps_{j+\ell}$ for large enough $j$.
\end{proof}
\begin{cor}\label{cor: dist within U and boundary U is zero}
Let $f\in \mc V$.
If $\phi=\phi^+$ a.e. and $\phi$ is an upper semi-continuous invariant graph, then $\phi=\phi^+$. In other words, $\phi^+$ is the unique upper semi-continuous invariant graph in its equivalence class.
Further,
\begin{align}\label{eq: pseudo continuity of phi +}
  \phi^+\left(\overline{
      B_r(\theta)}\right)\ssq\overline{\phi^+\left( B_r(\theta)\right)},
\end{align}
for all $\theta \in \T^d$ and all $r>0$.
\end{cor}
\begin{proof}
We first show \eqref{eq: pseudo continuity of phi +}. 
Let $\theta \in \T^d$ and $r>0$ be given and let $\theta_0 \in \d B_r(\theta)=\overline{B_r(\theta)}\setminus B_r(\theta)$.

Consider the case where $\theta_0 \notin \tilde \Omega_\infty$ and let $j$ be as
in Lemma~\ref{lem: theta NOT in tilde Omega infty}.  Equation~\eqref{eq: almost every point in a neighbourhood of Omegaj is in Omegaj} yields that for every $\rho>0$ there
is $\theta'\in B_r(\theta) \cap B_\rho(\theta_0)$ such that
$\theta'\in \Omega_j$.  Without loss of generality we may assume that
$j$ is large enough so that Proposition~\ref{prop: properties iterated
  boundary lines}\eqref{prop: properties iterated boundary lines iii}
gives
\begin{align*}
 \left|\phi^+_n(\theta_0)-\phi^+_n(\theta')\right|\leq L_j d(\theta_0,\theta')
\end{align*}
for arbitrary $n$ and thus
$\left|\phi^+(\theta_0)-\phi^+(\theta')\right|\leq L_j
d(\theta_0,\theta')\leq L_j \rho$ as $\phi^+_n \to \phi^+$ point-wise.
Sending $\rho$ to zero proves the statement in the case $\theta_0
\notin \tilde \Omega_\infty$.

Now, suppose $\theta_0 \in \tilde \Omega_\infty$ and let $\delta>0$.
Lemma~\ref{lem: theta in tilde Omega infty} yields that 
there is arbitrarily large $j\in \N$ such that $\theta_0
\in \Omega_j^{2K_{j+2}M_{j+2}}$. For sufficiently large $j$, equation \eqref{eq: theta in tilde omega infty then big
  neighbourhood is in omega j} gives $B_r(\theta)\cap B_{\delta
  \eps_{j}^{CK_{j-1}}}(\theta_0) \cap \Omega_{j}\neq \emptyset$, where
we may choose $C$ such that $L_{j}\leq\eps_{j}^{-CK_{j-1}}$ (see
Proposition~\ref{prop: properties iterated boundary lines} \eqref{prop:
  properties iterated boundary lines iii}).  Let $\theta'\in
B_r(\theta)\cap B_{\delta \eps_{j}^{CK_{j-1}}}(\theta_0) \cap
\Omega_{j}$.  Then
$\left|\phi^+_{2K_jM_j}(\theta_0)-\phi^+_{2K_jM_j}(\theta')\right|\leq\delta$
by Proposition~\ref{prop: properties iterated boundary lines} \eqref{prop: properties iterated boundary lines iii}.
Without loss of generality we may further assume that $j$ is large
enough to ensure
$\left|\phi^+(\theta_0)-\phi^+_{2K_jM_j}(\theta_0)\right|\leq\delta$ and
$\sum_{k=2K_jM_j}^\infty\alpha^{-\lam k}\leq\delta$, for $\lam$ as in Proposition~\ref{prop: properties iterated boundary lines}\eqref{prop: properties iterated boundary lines ii}.  This eventually
gives
\begin{align*}
 \left|\phi^+(\theta_0)-\phi^+(\theta')\right|\leq
\left|\phi^+(\theta_0)-\phi^+_{2K_jM_j}(\theta_0)\right|
+\left|\phi^+_{2K_jM_j}(\theta_0)-\phi^+_{2K_jM_j}(\theta')\right|
+\left|\phi^+_{2K_jM_j}(\theta')-\phi^+(\theta')\right|\leq3\delta, 
\end{align*}
where we used Proposition~\ref{prop: properties iterated boundary
  lines}\eqref{prop: properties iterated boundary lines ii} (again, assuming large enough $j$) to estimate
the last term.

Given arbitrary $\theta\in \T^d$ and $r>0$, we have thus shown that for each $\theta_0 \in \d B_r(\theta)$ there is a sequence $\theta_n\stackrel{n\to \infty}{\longrightarrow}\theta_0$ within $B_r(\theta)$ such that $\phi^+(\theta_0)=\lim_{n\to\infty} \phi^+(\theta_n)$. Hence, \eqref{eq: pseudo continuity of phi +} holds. In fact, the construction shows that even if $\phi=\phi^+$ only \emph{almost} everywhere, we still find a sequence ${\tilde\theta}_n\stackrel{n\to \infty}{\longrightarrow}\theta_0$ within $B_r(\theta)$ such that $\phi({\tilde\theta}_n)=\phi^+({\tilde\theta}_n)\stackrel{n\to \infty}{\longrightarrow}\phi^+(\theta_0)$. Thus, if $\phi$ is upper semi-continuous, this necessarily yields $\phi\geq \phi^+$. On the other hand, if $\phi$ is invariant, its graph is contained entirely within the maximal invariant set $\Lambda$ so that $\phi\leq\phi^+$. Thus, $\phi=\phi^+$.
\end{proof}
\begin{figure}
\centering 
\includegraphics[scale=1.3]{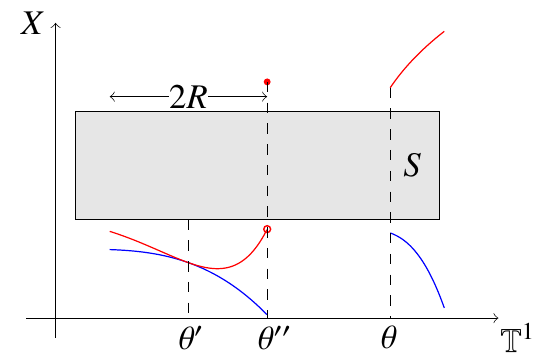}
\caption{The $1$-dimensional case: Assuming a gap within the minimal set implies the existence of a point $(\theta'',\phi^+(\theta''))$ which is isolated from one side (here, from the left). This contradicts Corollary~\ref{cor: dist within U and boundary U is zero}.}
\label{fig: filled in idea}
\end{figure}
Given an $f$-invariant and closed set $B\ssq \T^d\times \X$,
the associated \emph{upper} and \emph{lower bounding graphs}
\begin{align*}
 \phi^{+}_B(\theta)\=\sup\{x\: (\theta,x)\in B\} \quad \text{and} \quad
\phi^{-}_B(\theta)\=\inf\{x\: (\theta,x)\in B\} 
\end{align*}
are invariant graphs, where $\phi^+_B$ is upper semi-continuous and
$\phi^-_B$ is lower semi-continuous.
Vice versa, continuity of $f$
straightforwardly gives that the topological closure of an invariant
graph $\Phi$ is a closed invariant set. Further, if $\phi$ is upper
(lower) semi-continuous, then it equals the corresponding upper
(lower) bounding graph: $\phi=\phi^+_{\overline \Phi}$
($\phi=\phi^-_{\overline \Phi}$) (see \cite[Corollary~1 \& 2]{Stark}).
\begin{rem}
For the proof of the next statement, it is important to note that due to the non-zero Lyapunov exponents there is no lower
and upper semi-continuous invariant graph that coincides almost everywhere with $\phi^+$ 
and $\phi^-$, respectively (cf. \cite[Lemma~3.2]{jaeger:2003}). 
\end{rem}
\begin{thm}\label{thm: filled-in} Let $f\in \mc V$. Then
 $\left[\phi^-,\phi^+\right]$ is minimal. As a consequence, $D_B(\phi^-)=D_B(\phi^+)=d+1$.
\end{thm}
\begin{proof}
As $\phi^-$ and $\phi^+$ are lower and upper semi-continuous invariant graphs, respectively, $\left[\phi^-,\phi^+\right]$ is a compact invariant set.
 
For a contradiction, assume $\left[\phi^-,\phi^+\right]$ is not
minimal. Then there is a proper subset
$M\subset\left[\phi^-,\phi^+\right]$ which is compact and invariant.
Theorem~\ref{t.saddle-node} ($\mathcal{N}$) and Corollary~\ref{cor: dist within U and boundary U is zero} as well as the above remark yield
that $\phi_M^\pm=\phi^\pm$. Hence, there have to be $\theta \in \T^d$
and $x\in\left(\phi^-(\theta),\phi^+(\theta)\right)$ with
$(\theta,x)\notin M$.  Since $M$ is compact, there is an open strip
$S\=B_{\eps_1}(\theta_0)\times B_{\eps_2}(x_0)$ with $\eps_1,\eps_2>0$
centred at some $(\theta_0,x_0)\in\T^d\times \X$ such that
$(\theta,x)\in S$ and $S\cap M =\emptyset$.

By Theorem~\ref{t.saddle-node}, we may assume without loss of generality that there is a pinched point $\theta'\in B_{\eps_1}(\theta_0)$ with $\phi^-(\theta')=\phi^+(\theta')\leq x_0-\eps_2$. In other words, $\Phi^-$ and $\Phi^+$ have a common point below $S$. By continuity of $\phi^+$ at the pinched points (see the 
remark below Theorem~\ref{t.saddle-node}), we have that ${\left.\Phi^+\right |_{B_r(\theta')}}\=\Phi^+\cap B_r(\theta')\times[0,1]$ is below $S$ for all small enough $r>0$. Denote by $R$ the supremum of all such $r$ and suppose without loss of generality that $B_R(\theta')\ssq B_{\eps_1}(\theta')$. Then,
${\left.\Phi^+\right |_{B_R(\theta')}}$ is below $S$, while
${\left.\Phi^+\right |_{B_{R+\delta}(\theta')}}$ necessarily contains points above $S$ for each $\delta>0$. Hence, there is $\theta''\in \d B_R(\theta')$ such that
$(\theta'',\phi^+(\theta''))$ is above $S$, contradicting Corollary~\ref{cor: dist within U and boundary U is zero} (cf. Figure~\ref{fig: filled in idea}). This proves the desired
minimality. 

As an immediate consequence, we have $\overline {\phi^-}=\overline {\phi^+}=[\phi^-,\phi^+]$ and so, by the remark in Section~\ref{sec: hausdorff and box-counting Dimension}, $D_B(\phi^-)=D_B(\phi^+)=D_B([\phi^-,\phi^+])$. Since $\phi^-<\phi^+$ a.e., we further have $D_B([\phi^-,\phi^+])=d+1$.
\end{proof}

\section{Proof of Proposition~\ref{prop: properties iterated boundary lines}}
We now turn to the proof of Proposition~\ref{prop: properties iterated boundary lines}. It is based on both the $\mc C^2$-estimates and the dynamical assumptions that define the set $\widehat U_\w$ (see Section~\ref{subsec: basic setting and notation}).

A crucial point is to control the number of times a forward orbit spends in the contracting and a backward orbit spends in the expanding region, respectively.
For $n,N\in \N$ set
\begin{align*}
\mc P_n^N(\theta,x) &\= \# \{l \in [n,N-1]\cap \N_0\:f_{\theta}^l(x) \in C\text{ and } \theta+l\w\notin \I_0\};\\
\mc Q_n^N(\theta,x) &\= \# \{l \in [n,N-1]\cap \N_0\:f_{\theta}^{-l}(x) \in E\text{ and } \theta-l\w\notin \I_0+\w\}.
\end{align*}
The following combinatorial lemmas are important ingredients for this control. Their proofs can be found in \cite{fuhrmann2013NonsmoothSaddleNodesI}. In the following, it is convenient to set $M_{-1}\=0$ (as before) and $\I_{-1}\=\I_0$ as well as $\mc Z^-_{-1},\mc
Z^+_{-1}\=\emptyset$.
\begin{defn}
$(\theta,x)$ verifies $(\mc B1)_n$ and $(\mc B2)_n$, respectively if
\begin{enumerate}[$(\mc B1)_n$]
 \item $x \in C$ and $\theta \notin \mc Z^-_{n-1}$, \label{axiom: B1}
 \item $x \in E$ and $\theta \notin  \mc Z^+_{n-1}$.\label{axiom: B2}
\end{enumerate}
\end{defn}
\begin{lem}[{cf. \cite[Lemma~4.4]{fuhrmann2013NonsmoothSaddleNodesI}}]
\label{lem: duration of stay in contracting/expanding regions}
Let $f\in \mc V$, $n\in \N_0$ and assume $(\theta,x)$ satisfies
$(\mc B1)_n$.
Let $\mc L$ be the first time $l$
such that $\theta+l\w \in \I_n$
and let $0<\mc L_{1}<\ldots <\mc L_{N}=\mc L$ be all those 
times $m\leq \mc L$ for which $\theta+m \w \in \I_{n-1}$. Then
$f^{\mc L_{i}+M_{n-1}+2}(\theta,x)$ satisfies
$(\mc B1)_n$ for each $i=1,\ldots,N-1$ and the following implication holds
\begin{align*}
 f^k_\theta(x) \notin C \Rightarrow \theta+k\w \in \mc V_{n-1} \text{ and } f^k_\theta(x) \in [0,1]
\quad (k = 1,\ldots, \mc L).
\end{align*}
Analogously for backwards iteration: Instead of $(\mc B1)_n$, assume $(\theta,x)$ satisfies $(\mc B2)_n$.
Let $\mc R$ be the first time $r$ 
such that $\theta-r\w \in \I_n+\w$
and let $0<\mc R_{1}<\ldots <\mc R_{N}=\mc R$ be all those 
times $m\leq \mc R$ for which $\theta-m \w \in \I_{n-1}$.
Then $f^{-\mc R_{i}-M_{n-1}}(\theta,x)$ satisfies
$(\mc B2)_n$ for each $i=1,\ldots,N-1$ and the following implication holds
\begin{align*}
 f^{-k}_\theta(x) \notin E \Rightarrow \theta-k\w \in \mc W_{n-1} \text{ and } f_\theta^{-k}(x) \in [0,1] \quad (k= 1,\ldots, \mc R).
\end{align*}
\end{lem}

\begin{lem}[{cf. \cite[Lemma~4.8]{fuhrmann2013NonsmoothSaddleNodesI}}]\label{lem: estimate for times spent in contracting/expanding regions}
Let $f\in \mc V$ and assume $(\theta,x)$ verifies $(\mc B1)_{n}$ for $n\in \N$. Let $0<\mc L_1<\ldots <\mc L_N=\mc L$ be as in Lemma~\ref{lem: duration of stay in contracting/expanding regions}. Then, for each $i=1,\ldots,N$, we have
\begin{align}\label{eq: Lemma 4.8}
 P_k^{\mc L_i}(\theta,x) \geq b_n (\mc L_i-k) \quad (k=0,\ldots,\mc L_i-1).
\end{align}
Analogously, assume $(\theta,x)$ verifies 
$(\mc B\ref{axiom: B2})_{n}$ for $n\in \N$. Let $0<\mc R_1<\ldots <\mc R_N=\mc R$ be as in Lemma~\ref{lem: duration of stay in contracting/expanding regions}. Then, for each $i=1,\ldots,N$, we have
\begin{align*}
 Q_k^{\mc R_i}(\theta,x) \geq b_n (\mc R_i-k) \quad (k=1,\ldots,\mc R_i-1).
\end{align*}
\end{lem}
As before, we consider the iterated upper boundary lines only.
Given fixed $n\in \N$ and $\theta \in \T^d$, we set
\begin{align*}
\theta_k\=\theta-(n-k)\w \quad \text{and} \quad x_k\=f_{\theta_0}^k(1)
\end{align*}
such that
$\phi_k^+(\theta_k)=x_k$.

Let $p\in \N$ and consider a finite orbit $\{(\theta_0,x),\ldots, f^n(\theta_0,x)\}$ which initially verifies $(\mc B1)_p$ and hits $\mc I_p$ only at $\theta_0+n\w$. Lemma~\ref{lem: estimate for times spent in contracting/expanding regions} provides us with a lower bound on the times spent in the contracting region between
any time $k$ and only such following times at which the orbit hits $\I_{p-1}$.
If we want a lower bound on the times in the contracting region between any two consecutive moments $k<l$, we have to deal with the fact that Lemma~\ref{lem: duration of stay in contracting/expanding regions} might allow the orbit to stay in the expanding region for $M_{p-1}+1$ times after hitting $\I_{p-1}$. This is taken care of in the following corollary of Lemma~\ref{lem: duration of stay in contracting/expanding regions} and Lemma~\ref{lem: estimate for times spent in contracting/expanding regions}.

For $\theta\in \T^d$ and $0\leq k \leq n$, set 
\begin{align*}
p_k^n(\theta)=
\max\left\{p\in \N_{0}\:
\exists l \in \left[M_{p-1},\min\left\{n,n-k+M_{p}+1\right\}\right] \text{ such that }
\theta-l\w \in \I_p\right\}
\end{align*}
with $\max \emptyset \=-1$. 
At times, the following (and obviously equivalent) characterisation of $p_k^n(\theta)$ is useful
\begin{align*}
p_k^n(\theta)=
\max\left\{p\in \N_{0}\:
\exists l \in \left[\max\{0,k-M_{p}-1\},n-M_{p-1}\right] \text{ such that }
\theta_l \in \I_p\right\}.
\end{align*}
Observe that $p_\ell^n(\theta)$ and $p_{k-\ell}^{n-\ell}(\theta)$ are non-increasing in $\ell$.
\begin{cor}\label{cor: times in C}
Let $f\in \mc V$ and suppose
$(\theta_0,x)=(\theta-n\w,x)$ satisfies $(\B1)_{p_0^n(\theta)+1}$.
Then 
\begin{align}\label{eq: times in contracting region arbitrary successive times}
\mc P_{k}^{n}(\theta_0,x)\geq b_{p_k^n(\theta)+1} \left(n-k-\sum_{j=0}^{p_k^n(\theta)} (M_{j}+2)\right)\quad \text{for each }k=0,\ldots,n-1.
\end{align}
\end{cor}
\begin{proof}
For integers $p\geq-1$, set
\[
 \Theta_p \=\left\{(\theta,x,n)\in \T^d\times [c,1]\times \N\: p_0^n(\theta)\leq p \text{ and } (\theta-n\w,x) \text{ satisfy } (\mc B1)_{p_0^n(\theta)+1}\right\}.
\]
We say \emph{\eqref{eq: times in contracting region arbitrary successive times} holds within $\Theta_p$} if \eqref{eq: times in contracting region arbitrary successive times} is true for all $(\theta,x,n)\in\Theta_p$. We show by induction on $p$ that \eqref{eq: times in contracting region arbitrary successive times} holds within $\Theta_p$ for all $p$.
Note that within $\Theta_{-1}$ \eqref{eq: times in contracting region arbitrary successive times} follows directly from $(\ref{axiom: 2})$.

Suppose there is an integer $p_0\geq-1$ so that \eqref{eq: times in contracting region arbitrary successive times} holds within $\Theta_{p_0}$. 
Set $p=p_0+1$ and fix $(\theta,x,n)\in\Theta_p\setminus\Theta_{p_0}$ which is assumed to be non-empty without loss of generality.
Let $\mc L$ be the largest positive integer not bigger than $n-M_{p-1}$ such that
$\theta_{\mc L}\in \mc I_{p}$ and assume without loss of generality that $\mc L<n$. Note that
$p_{\mc L}^n(\theta)=p$.
First, let $k\in[\mc L,n-1]$. There are two cases to be considered.
\begin{enumerate}[(a)]
\item
Suppose $\mc L> n-M_p-2$. Then
$\mc L\in [\max\{0,k-M_p-1\},n-M_{p-1}]$ for all $k\leq n-1$, by definition of $\mc L$. 
Hence, $p_k^n(\theta)=p$ for all $k\in[\mc L,n-1]$ since $\theta_{\mc L}\in \I_p$. 
Thus, $k\geq \mc L> n-M_{p_k^n(\theta)}-2$ and so $M_{p_k^n(\theta)}> n-k-2$ so that \eqref{eq: times in contracting region arbitrary successive times} holds trivially. 
\item
Suppose $\mc L\leq n-M_p-2$. Without loss of generality, we may assume $n>\mc L+M_p+2$.
First, consider $k\geq\mc L+M_{p}+2$.
Then $p_k^n(\theta)<p$ and hence 
$p_{k-(\mc L+M_{p}+2)}^{n-(\mc L+M_{p}+2)}(\theta)\leq p_k^n(\theta)<p$.
Further by Lemma~\ref{lem: duration of stay in contracting/expanding regions}, $f^{\mc L+M_{p}+2}(\theta_0,x)$ satisfies $(\B1)_{p+1}$, and thus $(\B1)_{p_0+1}$.
Hence, we get 
\begin{align*}
&\mc P_{k}^{n}(\theta_0,x)=\mc P_{k-(\mc L+M_{p}+2)}^{n-(\mc L+M_{p}+2)}(f^{\mc L+M_{p}+2}(\theta_0,x))\geq b_{p_{k-(\mc L+M_{p}+2)}^{n-(\mc L+M_{p}+2)}(\theta)+1} \left(n-k-\sum_{j=0}^{p_{k-(\mc L+M_{p}+2)}^{n-(\mc L+M_{p}+2)}(\theta)} (M_{j}+2)\right)\\
&\geq
b_{p_k^n(\theta)+1} \left(n-k-\sum_{j=0}^{p_k^n(\theta)} (M_{j}+2)\right),
\end{align*}
where the first estimate follows by the induction hypothesis and the last estimate from the fact that $b_q$ is decreasing in $q$. Now, if $k\in [\mc L,\mc L+M_p+1]$, then
\begin{align*}
\mc P_{k}^{n}(\theta_0,x)&= \mc P_{k}^{\mc L+M_p+2}(\theta_0,x)+\mc P_{\mc L+M_p+2}^{n}(\theta_0,x) \geq \mc P_{\mc L+M_p+2}^{n}(\theta_0,x)\\
& \geq b_{p_{\mc L+M_p+2}^n(\theta)+1} \left(n-\mc L-M_p-2-\sum_{j=0}^{p_{\mc L+M_p+2}^n(\theta)} (M_{j}+2)\right)\\
&\geq b_{p_{k}^n(\theta)+1} \left(n-k-M_p-2-\sum_{j=0}^{p_{\mc L+M_p+2}^n(\theta)} (M_{j}+2)\right)\geq b_{p_{k}^n(\theta)+1} \left(n-k-\sum_{j=0}^{p_{k}^n(\theta)} (M_{j}+2)\right),
\end{align*}
where the last estimate holds since 
$p_k^n(\theta)=p$ for $k\leq \mc L+M_p+1$.
\end{enumerate}
We have thus shown
\begin{align}\label{eq: k in [L,n-1]}
 \mc P_{k}^{n}(\theta_0,x)\geq b_{p_{k}^n(\theta)+1} \left(n-k-\sum_{j=0}^{p_{k}^n(\theta)} (M_{j}+2)\right)
\end{align}
for $k\in[\mc L,n-1]$.

It remains to consider $k< \mc L$. Since $p_k^n(\theta)\geq p_{\mc L}^n(\theta)=p$, we obtain
\begin{align*}
 \mc P_{k}^{n}(\theta_0,x)&=\mc P_{k}^{\mc L}(\theta_0,x)+
\mc P_{\mc L}^{n}(\theta_0,x)\geq b_{p+1} (\mc L-k)+
b_{p_{\mc L}^n(\theta)+1} \left(n-\mc L-\sum_{j=0}^{p_{\mc L}^n(\theta)} M_{j}+2\right)\\
&\geq b_{p+1} \left(n-k-\sum_{j=0}^{p} M_{j}+2\right),
\end{align*}
where we used equation \eqref{eq: Lemma 4.8} and (\ref{eq: k in [L,n-1]}) in the first estimate.
As $(\theta,x,n)$ was arbitrary in $\Theta_p\setminus\Theta_{p_0}$, this shows that \eqref{eq: times in contracting region arbitrary successive times} holds within $\Theta_p$.
\end{proof}
For $k, n \in \N$, set $i_k^n\=\max\{l\:n-k\geq 2K_lM_l-M_l-1\}$.
\begin{prop}\label{prop: pk < ik}
Suppose $\theta \in \Omega_j^n$ for some $j\in \N$. Then $i_k^n\geq p_k^n(\theta)$ for all 
$0\leq k\leq n- (2K_{j-1}M_{j-1}-M_{j-1}-1)$.
\end{prop}
\begin{proof}
Note that by the assumptions $i_k^n\geq j-1$. 
Thus, without loss of generality we may assume $p_k^n(\theta)> j-1$.
By definition of $p_k^n(\theta)$, there is $l\in \left[M_{p_k^n(\theta)-1},n-k+M_{p_k^n(\theta)}+1\right]$ such that $\theta-l\w \in \mc I_{p_k^n(\theta)}$.
Since $\theta \in \Omega_j^n$, this implies
$l>2K_{p_k^n(\theta)}M_{p_k^n(\theta)}$ and thus,
$n-k>2K_{p_k^n(\theta)}M_{p_k^n(\theta)}-M_{p_k^n(\theta)}-1$ which means
$i_k^n\geq p_k^n(\theta)$.
\end{proof}

\begin{proof}[Proof of Proposition~\ref{prop: properties iterated boundary lines}]
Let $\theta \in \Omega_j^n$ and let $\mc L$ be the smallest positive integer such that $\theta_0- \mc L\w=\theta- (\mc L +n) \w \in \I_{p_n^n(\theta)}$.
Then $(\theta_0-(\mc L-1)\w,1)$ satisfies $(\mc B1)_{p_n^n(\theta)+1}$ because of $(\mc F1)_{p_n^n(\theta)}$.
By (\ref{eq:4}) and by the monotonicity \eqref{e.monotonicity} of the fibre maps, we have the implication
\begin{align*}
 f_{\theta_0-(\mc L-1)\w}^{\mc L-1+k}(1) \in C \implies f_{\theta_0}^k(1) \in C,
\end{align*}
for all $k\geq 0$. Further, we observe that
$p_0^n(\theta)=p_{\mc L-1}^{\mc L-1+n}(\theta)$ and actually 
$p_k^n(\theta)=p_{\mc L-1+k}^{\mc L-1+n}(\theta)$
for all $k=0,\ldots,n$.
By Corollary~\ref{cor: times in C}, we thus get
\begin{align}\label{eq: 20}
\begin{split}
 \mc P_{k}^{n}(\theta_0,1)&\geq \mc P_{\mc L-1+k}^{\mc L-1+n}(\theta_0-(\mc 
L-1)\w,1)\geq
b_{p_k^n(\theta)+1} \left(n-k-\sum_{\ell=0}^{p_k^n(\theta)} (M_{\ell}+2)\right)\\
& \stackrel{\text{Proposition~\ref{prop: pk < ik}}}{\geq}
b_{i_k^n+1} \left(n-k-\sum_{\ell=0}^{i_k^n} (M_{\ell}+2)\right),
\end{split}
\end{align}
for $0\leq k\leq n - (2K_{j-1}M_{j-1}-M_{j-1}-1)$. Now, note that 
$\sum_{\ell=0}^{i_k^n} (M_{\ell}+2)\leq 3/2 M_{i_k^n}$ for large enough
$i_k^n$ (and hence, for large enough $j$ since $i_k^n\geq j-1$). Further, $(n-k)/K_{i_k^n}\geq
2M_{i_k^n}-M_{i_k^n}/K_{i_k^n}-1/K_{i_k^n}$ by definition of
$i_k^n$. 
Thus for large enough $j$, we 
have $\sum_{\ell=0}^{i_k^n} (M_{\ell}+2)\leq (n-k)/K_{i_k^n}$ and so by \eqref{eq: 20}
\begin{align}\label{eq: P n-k n}
 \mc P_{k}^{n}(\theta_0,1)\geq b_{i^n_k+1} (1-1/K_{i_k^n}) (n-k) \geq b^2(n-k).
\end{align}

We hence have
\begin{align*}
&|\phi_{n}^+(\theta)-\phi_{n-1}^+(\theta)|\\
&=\phi^+_{n-1}(\theta)-\phi^+_{n}(\theta)=
\left(\phi^+_0(\theta_1)-\phi^+_1(\theta_1)\right) \cdot \prod_{k=1}^{n-1}
\frac{\phi^+_k(\theta_{k+1})-\phi^+_{k+1}(\theta_{k+1})}{\phi^+_{k-1}(\theta_{k})-\phi^+_{k}(\theta_{k})}\\
&\leq  \prod_{k=1}^{n-1}
\frac{f_{\theta_k}\left(\phi^+_{k-1}(\theta_{k})\right)-f_{\theta_k}\left(\phi^+_{k}(\theta_{k})\right)}{\phi^+_{k-1}(\theta_{k})-\phi^+_{k}(\theta_{k})}
\leq  \alpha^{p\left((n-1)-\mc P_{1}^{n}(\theta_0,1)\right)-2\mc P_{1}^n(\theta_0,1)/p}\\
&\stackrel{(\ref{eq: P n-k n})}{\leq} \alpha^{\left(p(1-b^2)-2b^2/p\right)(n-1)},
\end{align*}
where we assumed--without loss of generality--that $\phi^+_{k-1}(\theta_k)-\phi^+_{k}(\theta_k)>0$ for all $k\in\{1,\ldots,n\}$. 
This proves the first part with $\lam =2b^2/p-p(1-b^2)>0$.

Let $\wp_k^n(\theta,\theta')\=\#\left\{\ell\in[k,n-1]\cap \N_0\: x_\ell,x_\ell' \in C\right\}$ for
$\theta,\theta'\in \T^d$.
By induction on $n$, we first show that for all $n\in\N$
\begin{align}\label{eq: induction step n}
 |\phi^+_n(\theta)-\phi^+_n(\theta')|\leq S d(\theta,\theta') \sum_{k=1}^{n} \alpha^{p \left(n-k-\wp_{k}^n(\theta,\theta')\right)-2\wp_{k}^n(\theta,\theta')/p}.
\end{align}
For $n=1$, this is equation (\ref{eq: lipschitz theta}). Suppose \eqref{eq: induction step n} holds for some $n\in \N$. Since
$\wp_{k}^n(\theta-\w,\theta'-\w)+\wp_{n}^{n+1}(\theta,\theta')=\wp_{k}^{n+1}(\theta,\theta')$, this yields
\begin{align*}
\left|\phi_{n+1}^+(\theta)-\phi_{n+1}^+(\theta')\right|&=
\left|f_{\theta-\w}\left(\phi_{n}^+(\theta-\w)\right)-f_{\theta'-\w}\left(\phi_{n}^+(\theta'-\w)\right)\right|\\
&\leq
\alpha^{p\left(1-\wp_{n}^{n+1}(\theta,\theta')\right)-\frac2p\wp_{n}^{n+1}(\theta,\theta')}
|\phi_n^+(\theta-\w)-\phi_n^+(\theta'-\w)|+Sd(\theta-\w,\theta'-\w)\\
&\leq S d(\theta,\theta') \sum_{k=1}^{n+1}
\alpha^{p 
\left(n+1-k-\wp_{k}^{n+1}(\theta,\theta')\right)-2\wp_{k}^{n+1}(\theta,\theta')/p}
\end{align*}
where we used \eqref{eq: lipschitz x}--\eqref{eq: lipschitz x in C} in the first estimate and the induction hypothesis in the last step.
Hence, equation \eqref{eq: induction step n} holds.

Now, consider sufficiently large $j$ and suppose $\theta,\theta' \in \Omega_j^n$. Suppose $n >2K_{j-1}M_{j-1}-M_{j-1}-1$ and observe that
equation \eqref{eq: P n-k n} gives
\begin{align*}
\wp_{k}^{n}(\theta,\theta')
\geq n-k-(2(n-k)-\mc P_{k}^n(\theta)-\mc P_{k}^n(\theta')) \geq
n-k-2(1-b^2)(n-k)=(2b^2-1)(n-k)
\end{align*}
for all $k=0,\ldots,n - (2K_{j-1}M_{j-1}-M_{j-1}-1)$.
Plugging this into (\ref{eq: induction step n}) yields
\begin{align*}
 &|\phi^+_{n}(\theta)-\phi^+_{n}(\theta')|\\
&\leq S d(\theta,\theta') 
\left(\sum_{k=1}^{n-2K_{j-1}M_{j-1}-M_{j-1}-1}
\alpha^{\left(2p\left(1-b^2\right)-2(2b^2-1)/p\right)(n-k)}+
\sum_{k=n-2K_{j-1}M_{j-1}-M_{j-1}}^{n} \alpha^{p\left(n-k-\wp_{k}^{n}(\theta,\theta')\right)-2\wp_{k}^{n}(\theta,\theta')/p} 
\right)\\
&\leq L_j d(\theta,\theta'),
\end{align*}
where $L_j\=S \cdot 
\left(\sum_{l=2K_{j-1}M_{j-1}-M_{j-1}-1}^\infty
\alpha^{\left(2p\left(1-b^2\right)-2 
(2b^2-1)/p\right)l}+ 
\sum_{l=0}^{2K_{j-1}M_{j-1}-M_{j-1}}\alpha^{pl} 
\right)$. It is immediate that 
$|\phi^+_{n}(\theta)-\phi^+_{n}(\theta')|\leq L_j d(\theta,\theta')$
 holds for 
$n\leq2K_{j-1}M_{j-1}-M_{j-1}-1$, too.
Finally, observe that there is $C>0$ (independent of $j$) such that $L_j\leq\eps_j^{-CK_{j-1}}$ for large enough $j$.\qedhere
\end{proof}

\bibliography{Literaturnachweis_SNA,dynamics}{}
\bibliographystyle{unsrt}
\end{document}